\documentclass[reqno]{amsart}

\usepackage{amsmath}
\usepackage{amssymb}
\usepackage{amsfonts}
\usepackage{amsthm}
\usepackage[foot]{amsaddr}
\usepackage{mathtools}
\mathtoolsset{%
above-intertext-sep = -1ex
}

\usepackage[utf8]{inputenc}
\usepackage[T1]{fontenc}

\usepackage[sf,mono=false]{libertine}

\usepackage[sans]{dsfont}

\usepackage[%
cal=cm,
scr=euler,
]
{mathalfa}

\usepackage[dvipsnames,svgnames]{xcolor}
\colorlet{MyBlue}{DodgerBlue!75!Black}
\colorlet{MyGreen}{DarkGreen!95!Black}

\usepackage[font=small,labelfont=bf]{caption}
\usepackage{subfigure}
\usepackage{graphicx}
\graphicspath{{Figures/}} 
\usepackage{acronym}
\usepackage{latexsym}
\usepackage{paralist}
\usepackage{wasysym}
\usepackage{xspace}
\usepackage{framed}

\usepackage[numbers,sort&compress]{natbib}

\usepackage{hyperref}
\hypersetup{
colorlinks=true,
linktocpage=true,
pdfstartview=FitH,
breaklinks=true,
pdfpagemode=UseNone,
pageanchor=true,
pdfpagemode=UseOutlines,
plainpages=false,
bookmarksnumbered,
bookmarksopen=false,
bookmarksopenlevel=1,
hypertexnames=true,
pdfhighlight=/O,
urlcolor=MyBlue!60!black,linkcolor=MyBlue!70!black,citecolor=DarkGreen!70!black, 
pdftitle={},
pdfauthor={},
pdfsubject={},
pdfkeywords={},
pdfcreator={pdfLaTeX},
pdfproducer={LaTeX with hyperref}
}
\newcommand{\EMAIL}[1]{\email{\href{mailto:#1}{#1}}}

\numberwithin{equation}{section}  
\usepackage[sort&compress,capitalize,nameinlink]{cleveref}

\crefrangeformat{equation}{\upshape(#3#1#4)\textendash(#5#2#6)}

\newcommand{\dd}{\:d}

\newcommand{\eps}{\varepsilon}

\newcommand{\dif}{\dd}

\newcommand{\mgeq}{\succcurlyeq}

\DeclareMathOperator*{\argmin}{argmin}

\DeclareMathOperator{\dist}{dist}

\DeclareMathOperator{\tr}{tr}

\DeclareMathOperator{\Id}{Id}

\newcommand{\eqdef}{\triangleq}

\newcommand{\const}{\mathtt{c}}
\newcommand{\ta}{\mathtt{a}}

\newcommand{\scrA}{\mathcal{A}}
\newcommand{\scrB}{\mathcal{B}}

\newcommand{\scrE}{\mathcal{E}}
\newcommand{\scrF}{\mathcal{F}}
\newcommand{\scrG}{\mathcal{G}}

\newcommand{\scrI}{\mathcal{I}}
\newcommand{\scrJ}{\mathcal{J}}

\newcommand{\scrL}{\mathcal{L}}

\newcommand{\scrQ}{\mathcal{Q}}

\newcommand{\scrX}{\mathcal{X}}

\renewcommand{\Pr}{\mathbb{P}}
\newcommand{\Ex}{\mathbb{E}}

\newcommand{\measP}{{\mathsf{P}}}

\newcommand{\F}{{\mathbb{F}}}

\newcommand{\0}{\mathbf{0}}
\newcommand{\1}{\mathbf{1}}

\newcommand{\R}{\mathbb{R}}

\newcommand{\N}{\mathbb{N}}

\newcommand{\Lim}{\mathsf{Lim}}

\DeclareMathOperator{\VI}{VI}
\DeclareMathOperator{\SVI}{SVI}

\usepackage{algorithm}
\usepackage{algpseudocode}								

\theoremstyle{plain}
\newtheorem{theorem}{Theorem}
\newtheorem{corollary}[theorem]{Corollary}
\newtheorem*{corollary*}{Corollary}
\newtheorem{lemma}[theorem]{Lemma}
\newtheorem{proposition}[theorem]{Proposition}

\theoremstyle{definition}
\newtheorem{definition}[theorem]{Definition}
\newtheorem*{definition*}{Definition}
\newtheorem{assumption}{Assumption}

\theoremstyle{remark}
\newtheorem{remark}{Remark}
\newtheorem*{remark*}{Remark}
\newtheorem*{notation*}{Notational remark}
\newtheorem{example}{Example}
\newtheorem{experiment}{Experiment}

\numberwithin{theorem}{section}
\numberwithin{remark}{section}
\numberwithin{example}{section}

\DeclarePairedDelimiter{\abs}{\lvert}{\rvert}

\DeclarePairedDelimiter{\inner}{\langle}{\rangle}
\DeclarePairedDelimiter{\norm}{\lVert}{\rVert}

\begin{document}


\title[Stochastic Forward-Backward-Forward]{Forward-Backward-Forward Methods with Variance Reduction for Stochastic Variational inequalities}

\author[R.I. Bo\c t]{Radu Ioan Bo\c t$^{\star}$}
\address{$^{\star}$\,%
Faculty of Mathematics, University of Vienna, Oskar-Morgenstern-Platz 1, A-1090 Vienna, Austria.}
\EMAIL{radu.bot@univie.ac.at}
\author[P.~Mertikopoulos]{Panayotis Mertikopoulos$^{\circ}$}
\address{$^{\circ}$\,%
Univ. Grenoble Alpes, CNRS, Inria, Grenoble INP, LIG 38000 Grenoble, France.}
\EMAIL{panayotis.mertikopoulos@imag.fr}
\author[M.Staudigl]{\\Mathias Staudigl$^{\diamond}$}
\address{$^{\diamond}$\,%
Maastricht University, Department of Quantitative Economics, P.O. Box 616, NL\textendash 6200 MD Maastricht, The Netherlands.}
\EMAIL{m.staudigl@maastrichtuniversity.nl}
\author[P.T. Vuong]{Phan Tu Vuong$^{\star}$}
\EMAIL{vuong.phan@univie.ac.at}

\thanks{%
R.~I.~Bo\c{t} and P.~T.~Vuong acknowledge support by the Austrian Science Fund (FWF) within the project I2419-N32 (``Employing Recent Outcomes in Proximal Theory Outside the Comfort Zone'').
P.~Mertikopoulos has received financial support from
the FMJH Program PGMO under grant HEAVY.NET
and
the French National Research Agency (ANR) under grant ORACLESS (ANR\textendash 16\textendash CE33\textendash 0004\textendash 01).
M.~Staudigl and P.~Mertikopoulos have been sponsored by the COST Action CA16228 ``European Network for Game Theory''.}


\newcommand{\acli}[1]{\textit{\acl{#1}}}
\newcommand{\acdef}[1]{\textit{\acl{#1}} \textup{(\acs{#1})}\acused{#1}}
\newcommand{\acdefp}[1]{\emph{\aclp{#1}} \textup(\acsp{#1}\textup)\acused{#1}}

\newacro{EV}{expected value}
\newacro{SAA}{sample average approximation}
\newacro{SO}{stochastic oracle}
\newacro{SA}{stochastic approximation}
\newacro{FB}{forward-backward}
\newacro{FBF}{forward-backward-forward}
\newacro{SFBF}{stochastic forward-backward-forward}
\newacro{SEG}{stochastic extra-gradient}
\newacro{UBV}{uniformly bounded variance}
\newacroplural{NE}[NE]{Nash equilibria}
\newacro{VI}{variational inequality}
\newacroplural{VI}[VIs]{variational inequalities}
\newacro{iid}[i.i.d.]{independent and identically distributed}
\newacro{EE}{energy efficiency}

\begin{abstract}
We develop a new stochastic algorithm with variance reduction for solving pseudo-monotone stochastic variational inequalities. Our method builds on Tseng's \ac{FBF} algorithm, which is known in the deterministic literature to be a valuable alternative to Korpelevich's extragradient method when solving variational inequalities over a convex and closed set governed by pseudo-monotone, Lipschitz continuous operators. The main computational advantage of Tseng's algorithm is that it relies only on a single projection step and two independent queries of a stochastic oracle. Our algorithm incorporates a variance reduction mechanism and leads to almost sure (a.s.) convergence to an optimal solution. To the best of our knowledge, this is the first stochastic look-ahead algorithm achieving this by using only a single projection at each iteration.
\end{abstract}
\maketitle

\renewcommand{\sharp}{\gamma}
\acresetall
\allowdisplaybreaks

\section{Introduction}
\label{sec:introduction}

In this paper we consider the following variational inequality problem, denoted as $\VI(T,\scrX)$, or simply  $\VI$:  given a nonempty closed and convex set $\scrX\subseteq\mathbb{R}^{d}$ and a single valued map $T:\mathbb{R}^{d}\to\mathbb{R}^{d}$, find $x^{\ast}\in\scrX$ such that
\begin{equation}
\label{eq:strong}
\inner{T(x^{\ast}),x-x^{\ast}}
	\geq 0
	\quad
	\text{for all $x\in\scrX$}.
\end{equation}
We call $S(T,\scrX)\equiv\scrX_{\ast}$ the set of (Stampacchia) solutions of $\VI(T,\scrX)$. The variational inequality problem \eqref{eq:strong} arises in many interesting applications in economics, game theory and engineering \cite{Scu10,RavSha11,KanShan12,JNT11,MerSta18b}, and includes as a special case first-order optimality conditions for nonlinear optimization, by choosing $T=\nabla f$ for some smooth function $f$. If $\scrX$ is unbounded, it can also be used to formulate complementarity problems, systems of equations, saddle point problems and many equilibrium problems. We refer the reader to \cite{FacPan03} for an extensive review of applications in engineering and economics.

In many instances the problem $\VI$ arises as the expected value of an underlying stochastic optimization problem whose primitives are defined on a probability space $(\Omega,\scrF,\Pr)$ carrying a random variable $\xi:(\Omega,\scrF)\to(\Xi,\scrA)$ taking values in a measurable space $(\Xi,\scrA)$ and inducing a law $\measP=\Pr\circ\xi^{-1}$. Given the random element $\xi$, consider the measurable mapping $F:\scrX\times\Xi\to\R^{d}$, defining an integrable random vector $F(x,\xi):\Omega\to\R^{d}$ via the composition $F(x,\xi)(\omega)=F(x,\xi(\omega))$. The stochastic variational inequality problem on which we will focus in this paper is denoted by  $\SVI$ and defined as follows: 
\begin{definition}\label{def:SVI}
Let the operator $T:\mathbb{R}^{d}\to\mathbb{R}^{d}$ be defined by
\begin{equation}\label{eq:T}
T(x):=\Ex_{\xi}[F(x,\xi)]:=\int_{\Omega}F(x,\xi(\omega))\dif\Pr(\omega)=\int_{\Xi}F(x,z)\dif\measP(z).
\end{equation}
Find $x^{\ast}\in\scrX$ satisfying \eqref{eq:strong}.
\end{definition}

This definition is known as the \emph{expected value formulation of the stochastic variational inequality problem}. The expected value formulation goes back to the seminal work of \cite{KinRoc93}. By its very definition, if the operator $T$ defined in \eqref{eq:T} would be known, then the expected value formulation can be solved by any standard solution technique for deterministic variational inequalities. However, in practice, the operator $T$ is usually not directly accessible, either due to excessive computations involved in performing the integral, or because $T$ itself is the solution of an embedded subproblem. Hence, in most situations of interest, the solution of $\SVI$ relies on random samples of the operator $F(x,\xi)$. In this context, there are two current methodologies available; the \ac{SAA} approach replaces the expected value formulation with an empirical estimator of the form 
\begin{align*}
\hat{T}^{N}(x)=\frac{1}{N}\sum_{j=1}^{N}F(x,\xi_{j}),
\end{align*}
and use the resulting deterministic map $T^{N}$ as the input in one existing algorithm of choice. We refer to \cite{DenRusSha09} for this solution approach in connection with Monte Carlo simulation. We note that this approach is the standard choice in expected residual minimization problems, when $\measP$ is unknown but accessible via a Monte Carlo approach. 

A different methodology is the \ac{SA} approach, where samples are obtained in an online fashion, namely, the decision maker chooses one deterministic algorithm to solve the expected value formulation, and draws a fresh random variable whenever needed. The mechanism to draw a fresh sample from $\measP$ is usually named a \ac{SO}, which report generates a stochastic error $F(x,\xi)-T(x)$.  

Until very recently, the SA approach has only been used for the expected value formulation under very restrictive assumptions. To the best of our knowledge, the first formulation of an SA approach for a stochastic VI problem was made by \cite{JiaXu09}, under the assumption of strong monotonicity and continuity of the operator $T$. There, a proximal point algorithm of the form
\begin{equation}\label{eq:prox}
X_{n+1}=\Pi_{\scrX}[X_{n}+\alpha_{n}F(x_{n},\xi_{n})]
\end{equation}
is considered, where $\Pi_{\scrX}$ denotes the Euclidean projection onto $\scrX$, $(\xi_{n})_{n\geq 0}$ is a sample of $\measP$, and $(\alpha_{n})_{n\geq 0}$ is a sequence of positive step sizes. Almost sure convergence of the iterates is proven for small step sizes, assuming $T$ is Lipschitz continuous and strongly monotone, and the stochastic error is uniformly bounded. Relaxing strong monotonicity to plain monotonicity, the recent paper \cite{YouNedSha17} incorporated a Tikhonov regularization scheme into the stochastic approximation algorithm \eqref{eq:prox} and proved almost sure convergence of the generated stochastic process. The only established method guaranteeing almost sure convergence under the significantly weaker assumption of \emph{pseudo-monotonicty} of the mean operator is the extragradient approach of \cite{IusJofOliTho17}. The original Korpelevich extragradient scheme of \cite{Kor76} consists of two projection steps using two evaluations of the deterministic map $T$ at generated test points $y_{n}$ and $x_{n}$. Extending this to the stochastic oracle case, we arrive at the \ac{SEG} method 
\begin{equation}
\label{eq:SEG}
\tag{SEG}
\begin{aligned}
Y_{n}
	&= \Pi_{\scrX}[X_{n}-\alpha_{n}A_{n+1}]
	\\
X_{n+1}
	&= \Pi_{\scrX}[X_{n}-\alpha_{n}B_{n+1}]
\end{aligned}
\end{equation}
where $(A_{n})_{n\geq 1},(B_{n})_{n\geq 1}$ are stochastic estimators of $T(X_{n})$, and $T(Y_{n})$, respectively.  The paper \cite{IusJofOliTho17} constructs these estimators by relying on a dynamic sampling strategy, where noise reduction of the estimators is achieved via a mini-batch sampling of the stochastic operators $F(X_{n},\xi$) and $F(Y_{n},\xi)$. Within this mini-batch formulation, almost sure convergence of the stochastic process $(X_{n})_{n\in\N}$ to the solution set can be proven even with constant step size implementations of \ac{SEG}. On top, optimal convergence rates of $O(1/N)$ in terms of the mean squared residual of the $\VI$ are obtained.

\subsection{Our Contribution}
We briefly summarize the main contributions of this work. The most costly part of \ac{SEG} are the two separate projection steps performed at each single iteration of the method. We show in this paper that a stochastic version of Tseng's forward-backward-forward \cite{Tse00}, which we call the \ac{SFBF} algorithm, preserves the strong trajectory-based convergence results, while the saving of one projection step allows us to beat \ac{SEG} significantly in terms of computational overheads and runtimes. In terms of convergence properties the \ac{SFBF} algorithm developed in this  paper has the same good properties as \ac{SEG}. However, \ac{SFBF} is potentially more efficient than \ac{SEG} in each iteration since it relies only on a single euclidean projection step. The price to pay for this is that we obtain an infeasible method (as is typical for primal-dual schemes) with a lower computational complexity count at the positive side. Additionally, the theoretically allowed range for step sizes is by the constant factor $\sqrt{3}$ times larger than the theoretically allowed largest step size in \ac{SEG}. This constant factor gain results in significant improvements in terms of the convergence speed. This will be illustrated with extensive numerical evidences reported in \cref{sec:numerics}.

\section{Preliminaries}
\label{sec:prelims}
\subsection{Notation}
For $x,y\in\mathbb{R}^{d}$, we denote by $\inner{x,y}$ the standard inner product, and by $\norm{x}\equiv\norm{x}_{2}:=\inner{x,x}^{\frac{1}{2}}$ the corresponding norm. For $p\in[1,\infty]$, the $\ell_{p}$ norm on $\R^d$ is defined for $x=(x_1,...,x_p)$ as $\norm{x}_{p}:=\left(\sum_{i=1}^{n}\abs{x_{i}}^{p}\right)^{\frac{1}{p}}$. For a nonempty, closed and convex set $E\subseteq\mathbb{R}^{d}$, the Euclidean projector is defined as $\Pi_{E}(x):=\argmin_{y\in E}\norm{y-x}$ for $x\in\mathbb{R}^{d}$. All random elements are defined on a given probability space $(\Omega,\scrF,\Pr)$. An $E$-valued random variable is a $(\scrF,\scrE)$-measurable mapping $f:\Omega\to E$; we write $f\in L^{0}(\Omega,\scrF,\Pr;E)$. For every $p\in[1,\infty]$, define the equivalence class of random variables $f\in L^{0}(\Omega,\scrF,\Pr;E)$ with $\Ex(\norm{f}^{p})^{1/p}<\infty$ as $L^{p}(\Omega,\scrF,\Pr;E)$. If $\scrG\subseteq\scrF$, the conditional expectation of the random variable $f\in L^{p}(\Omega,\scrF,\Pr;E)$ is denoted by $\Ex[f\vert\scrG]$. For $f_{1},\ldots,f_{k}\in L^{p}(\Omega,\scrF,\Pr;E)$ we denote the sigma-algebra generated by these random variables by $\sigma(f_{1},\ldots,f_{k})$, this is the smallest sigma-algebra measuring the random variables $f_{1},\ldots, f_{k}$. Let $(\Omega,\scrF,\F=(\scrF_{n})_{n\geq 0},\Pr)$ be a complete stochastic basis. We denote by $\ell^{0}(\F)$ the set of random sequences $(\xi_{n})_{n\geq 1}$ such for each $n\in\N$, $\xi_{n}\in L^{0}(\Omega,\scrF_{n},\Pr;\R)$. For $p\in[1,\infty]$, we set 
\begin{align*}
\ell^{p}(\F)\eqdef \{(\xi_{n})_{n\geq 1}\in\ell^{0}(\F)\vert \sum_{n\geq 1}\abs{\xi_{n}}^{p}<\infty\quad\Pr\text{-a.s.}\}. 
\end{align*}
The following properties of the euclidean projection on a closed and convex set are well known. 
\begin{lemma}\label{lem:projection}
Let $K\subseteq\mathbb{R}^{d}$ be a nonempty, closed and convex set. Then:
\begin{itemize}
\item[(i)] $\Pi_{K}(x)$ is the unique point of $K$ satisfying $\inner{x-\Pi_{K}(x),y-\Pi_{K}(x)}\leq 0$ for all $y\in K$;
\item[(ii)] for all $x\in\mathbb{R}^{d}$ and $y\in K$, we have $\norm{\Pi_{K}(x)-y}^{2}+\norm{\Pi_{K}(x)-x}^{2}\leq\norm{x-y}^{2}$;
\item[(iii)] for all $x,y\in\mathbb{R}^{d}$, $\norm{\Pi_{K}(x)-\Pi_{K}(y)}\leq\norm{x-y}$;
\item[(iv)] given $\alpha>0$ and $T:K\to\mathbb{R}^{d}$,  the set of solutions of the variational problem $\VI(T,K)$ can be expressed as $S(T,K)=\{x\in\mathbb{R}^{d}\vert x=\Pi_{K}(x-\alpha T(x))\}$.  
\end{itemize}
\end{lemma}

\begin{remark}
In the literature on variational inequalities, there exists an alternative solution concept known as \emph{weak}, or \emph{Minty}, solutions. In this paper we are only interested in \emph{strong}, or \emph{Stampacchia}, solutions of $\VI(T,K)$, defined by inequality \eqref{eq:strong}.
\end{remark}
 
Another useful fact we use in this paper is the following elementary identity. 
\begin{lemma}[Pythagorean identity]
\label{lem:Pyth}
For all $x,x_{n},x_{n+1}\in\mathbb{R}^{d}$ we have
\begin{align*}
\norm{x_{n+1}-x}^{2}+\norm{x_{n+1}-x_{n}}^{2}-\norm{x_{n}-x}^{2}=2\inner{x_{n+1}-x_{n},x_{n+1}-x}.
\end{align*}
\end{lemma}
\subsection{Probabilistic Tools}
We recall the Minkowski inequality: for given functions $f,g\in L^{p}(\Omega,\scrF,\Pr;E),\scrG\subseteq\scrF$ and $p\in[1,\infty]$, we have  
\begin{equation}
\label{eq:Minkowski}
\Ex[\norm{f+g}^{p}\vert\scrG]^{1/p}\leq \Ex[\norm{f}^{p}\vert\scrG]^{1/p}+\Ex[\norm{g}^{p}\vert\scrG]^{1/p}.
\end{equation}

For the convergence analysis we will make use of the following classical lemma (see e.g. \cite[Lemma 11, page 50]{Pol87}).
\begin{lemma}[Robbins-Siegmund]
\label{lem:RS}
Let $(\Omega,\scrF,\F=(\scrF_{n})_{n\geq 0},\Pr)$ be a discrete stochastic basis. Let $(v_{n})_{n\geq 1},(u_{n})_{n\geq 1}\in\ell^{0}_{+}(\F)$ and $(\theta_{n})_{n\geq 1},(\beta_{n})_{n\geq 1}\in\ell^{1}_{+}(\F)$ be such that  for all $n \geq 0$
\begin{align*}
\Ex[v_{n+1}\vert\scrF_{n}]\leq(1+\theta_{n})v_{n}-u_{n}+\beta_{n}\qquad  \Pr-\text{a.s. }.
\end{align*}
Then $(v_{n})_{n\geq 0}$ converges a.s. to a random variable $v$, and $(u_{n})_{n\geq 1}\in\ell^{1}_{+}(\F)$.
\end{lemma}
Finally, we need the celebrated Burkholder-Davis-Gundy inequality (see e.g. \cite{Str11}).
\begin{lemma}\label{lem:BDG}
Let $(\Omega,\scrF,(\scrF_{n})_{n\geq 0},\Pr)$ be a discrete stochastic basis and $(U_{n})_{n\geq 0}$ a vector-valued martingale relative to this basis. Then, for all $p\in[1,\infty)$, there exists a universal constant $C_{p}>0$ such that for every $N \geq 1$
\begin{equation*}
\Ex\left[\left(\sup_{0 \leq i\leq N}\norm{U_{i}}\right)^{p}\right]^{1/p}\leq C_{p}\Ex\left[\left(\sum_{i=1}^{N}\norm{U_{i}-U_{i-1}}^{2}\right)^{p/2}\right]^{\frac{1}{p}}.
\end{equation*}
\end{lemma}
When combined with Minkowski inequality, we obtain for all $p\geq 2$ a constant $C_{p}>0$ such that for every $N \geq 1$
\begin{equation*}
\Ex\left[\left(\sup_{0 \leq i\leq N}\norm{U_{i}}\right)^{p}\right]^{1/p}\leq C_{p}\sqrt{\sum_{k=1}^{N}\Ex\left(\norm{U_{i}-U_{i-1}}^{p}\right)^{2/p}}
\end{equation*}
\section{The stochastic forward-backward-forward algorithm}
\label{sec:algorithm}

In this paper we study a forward-backward-forward algorithm of Tseng type under weak monotonicity assumptions. The blanket hypotheses we consider throughout our analysis are summarized here:
\begin{assumption}[Consistency]
\label{ass:Consistent}
The solution set $\scrX_{\ast}\equiv S(T,\scrX)$ is nonemtpy.
\end{assumption}
\begin{assumption}[Stochastic Model]
\label{ass:stochastic}
The set $\scrX\subseteq\mathbb{R}^{d}$ is nonempty, closed  and convex, $(\Xi,\scrA)$ is a measurable space and $F:\mathbb{R}^{d}\times\Xi\to\mathbb{R}^{d}$ is a Carath\'{e}odory map.\footnote{The mapping $x\mapsto F(x,\xi)$ is continuous for a. e. $\xi\in\Xi$, and $\xi\mapsto F(x,\xi)$ is measurable for all $x\in\mathbb{R}^{d}$;  $\xi$ is a random variable with values in $\Xi$, defined on a probability space $(\Omega,\scrF,\Pr)$.}
\end{assumption}
\begin{assumption}[Lipschitz continuity]
\label{ass:Lipschitz}
The averaged operator $T(\cdot)=\Ex_{\xi}[F(\cdot,\xi)]:\mathbb{R}^{d}\to\mathbb{R}^{d}$ is Lipschitz continuous with modulus $L>0$.
\end{assumption}
\begin{assumption}[Pseudo-Monotonicity]
\label{ass:monotone}
The averaged operator $T(\cdot)=\Ex_{\xi}[F(\cdot,\xi)]$ is pseudo-monotone on $\mathbb{R}^{d}$, which means
\[
\forall x,y\in\R^{d}: \inner{T(x),y-x}\geq 0\Rightarrow \inner{T(y),y-x}\geq 0.
\]
\end{assumption}
At each iteration, the decision maker has access to a stochastic oracle, reporting an approximation of $T(x)$ of the form
\begin{equation}\label{eq:approx}
\hat{T}_{n+1}(x,\xi_{n+1})\eqdef \frac{1}{m_{n+1}}\sum_{i=1} ^{m_{n+1}}F(x,\xi^{(i)}_{n+1}) \ \mbox{for} \ x\in\mathbb{R}^{d}.
\end{equation}
The sequence $(m_{n})_{n\geq 1}\subseteq\N$ determines the batch size of the stochastic oracle. The random sequence $\xi_{n}=(\xi^{(1)}_{n},\ldots,\xi^{(m_{n})}_{n})$ is an i.i.d draw from $\measP$. Approximations of the form \eqref{eq:approx} are very common in Monte-Carlo simulation approaches, machine learning and computational statistics (see e.g. \cite{AtcForMou17,BotCurNoc18}, and references therein); they are easy to obtain in case we are able to sample from the measure $\measP$. The forward-backward-forward algorithm requires two queries from the stochastic oracle in which mini-batch estimators of the averaged map $T$ are revealed. This dynamic sampling strategy requires a sequence of integers $(m_{n})_{n\geq 1}$ (the \emph{batch size}) determining the size of the data set to be processed at each iteration. The random sample on each mini-batch consists of two independent stochastic processes $\xi_{n}$ and $\eta_{n}$ drawn from the law $\measP$, and explicitly given by
\begin{align*}
\xi_{n}\eqdef (\xi^{(1)}_{n},\ldots,\xi^{(m_{n})}_{n})\text{ and }\eta_{n}\eqdef (\eta^{(1)}_{n},\ldots,\eta^{(m_{n})}_{n})\qquad\forall n\geq 1.
\end{align*}
Given the current position $X_{n}$, Algorithm \ac{SFBF} queries the \ac{SO} once, to obtain the estimator $A_{n+1}\eqdef \hat{T}_{n+1}(X_{n},\xi_{n+1})$, and then constructs the random variable $Y_{n}=\Pi_{\scrX}(X_{n}-\alpha_{n}A_{n+1})$. Next, a second query to \ac{SO} is made to obtain the estimator $B_{n+1}\eqdef \hat{T}_{n+1}(Y_{n},\eta_{n+1})$, followed by the update $X_{n+1}=Y_{n}+\alpha_{n}(A_{n+1}-B_{n+1})$. The pseudocode for \ac{SFBF} is given in \cref{alg:TsengFB}.

\begin{algorithm}[htbp]
\caption{Stochastic forward-backward-forward (SFBF)}
\label{alg:TsengFB}
\tt
\begin{algorithmic}[1]
\Require
step-size sequence $\alpha_{n}$;
batch size sequence $m_{n}$
\State
Initialize $X$
	\Comment{initialization}%
\For{$n = 1,2,\dotsc$}
\State
Draw samples $\xi^{i}$ and $\eta^{i}$ from $\measP$ ($i = 1,\dotsc,m_{n}$)%
	\vspace{1ex}
	\State
	Oracle returns $\displaystyle A \leftarrow \frac{1}{m_{n}} \sum_{i=1}^{m_{n}} F(X,\xi^{i})$
		\Comment{first oracle query}\vspace{.5ex}
	\State
		Set $Y \leftarrow \Pi_{\scrX}(X - \alpha_{n}A)$
		\Comment{\acl{FB} step}\vspace{.5ex}
	\State
		Oracle returns 
		\(
		\displaystyle
		B
			\leftarrow \frac{1}{m_{n}} \sum_{i=1}^{m_{n}} F(Y,\eta^{i})
		\)
		\Comment{second oracle query}\vspace{.5ex}
	\State
		Set $X \leftarrow Y + \alpha_{n} (A - B)$
		\Comment{second forward step}
\EndFor
\end{algorithmic}
\end{algorithm}
Observe that Algorithm \ac{SFBF} is an infeasible method: the iterates $(X_{n})_{n\geq 0}$ are not necessarily elements of the admissible set $\scrX$, but the process $(Y_{n})_{n\geq 0}$ is by construction so. In the stochastic optimization case, i.e. for instances where $A_{n+1}$ is an unbiased estimator of the gradient of a real-valued function, the process $(Y_{n})_{n\geq 0}$ is seen to be a projected gradient step, where $A_{n+1}$ acts as an unbiased estimator for the stochastic gradient. This gradient step is used in an extrapolation step to generate the iterate $X_{n+1}$. We just mention that related popular primal-dual splitting schemes like ADMM \cite{BoyChuEckADMM11,CheCheOuPas18} are infeasible by nature as well. 

\begin{assumption}[Step-size choice]
\label{ass:step}
The step-size sequence $(\alpha_{n})_{n\geq 0}$ in Algorithm \ac{SFBF} satisfies
\begin{equation*}
0<\underline{\alpha}\eqdef\inf_{n\geq 0}\alpha_{n}\leq\bar{\alpha}\eqdef\sup_{n\geq 1}\alpha_{n}<\frac{1}{\sqrt{2}L}.
\end{equation*}
\end{assumption}
For $n\geq 0$, we introduce the \emph{approximation error}
\begin{equation}\label{eq:Doob}
W_{n+1}\eqdef A_{n+1}-T(X_{n}),\text{ and } Z_{n+1}\eqdef B_{n+1}-T(Y_{n}),
\end{equation}
and the sub-sigma algebras $(\scrF_{n})_{n\geq 0},(\hat{\scrF}_{n})_{n\geq 0}$, defined by $\scrF_{0}\eqdef\sigma(X_{0})$, and
\begin{align*}
\scrF_{n}&\eqdef \sigma(X_{0},\xi_{1},\xi_{2},\ldots,\xi_{n},\eta_{1},\ldots,\eta_{n})\qquad \forall n\geq 1,
\end{align*}
and
\begin{align*}
\hat{\scrF}_{n} & \eqdef \sigma(X_{0},\xi_{1},\ldots,\xi_{n},\xi_{n+1},\eta_{1},\ldots,\eta_{n}) \qquad\forall n\geq 0,
\end{align*}
respectively. Observe that $\scrF_{n}\subseteq\hat{\scrF}_{n}$ for all $n\geq 0$. We also define the filtrations $\F\eqdef (\scrF_{n})_{n\geq 0}$ and $\hat{\F}\eqdef (\hat{\scrF}_{n})_{n\geq 0}$. The introduction of these two different sub-sigma algebras is important for many reasons. First, observe that they embody the information the learner has about the optimization problem. Indeed, the sub-sigma algebra $(\scrF_{n})_{n \geq 0}$ corresponds to the information the decision maker has at the beginning the $n$-th iteration, whereas $(\hat{\scrF}_{n})_{n \geq 0}$ is the information the decision maker has after the first (projection)-step of the iteration. Therefore, $(Y_{n})_{n \geq 0}$ is measurable with respect to the sub-sigma algebra $(\hat{\scrF}_{n})_{n \geq 0}$ and $(X_{n})_{n \geq 0}$ is measurable with respect to the sub-sigma algebra $(\scrF_{n})_{n \geq 0}$. Second, we see that the process $(W_{n})_{n\geq 1}$ is $\F$-adapted, whereas the process $(Z_{n})_{n\geq 1}$ is $\hat{\F}$-adapted, unbiased approximations relative to the respective information structures are provided: 
\begin{align*}
\Ex[W_{n+1}\vert\scrF_{n}]=0 \text{ and } \Ex[Z_{n+1}\vert \hat{\scrF}_{n}]=0 \ \forall n \geq 0.
\end{align*}
\begin{assumption}[Batch Size]
\label{ass:batch}
The batch size sequence $(m_{n})_{n\geq 1}$ satisfies $\sum_{n=1}^{\infty}\frac{1}{m_{n}}<\infty$. 
\end{assumption}
A sufficient condition on the sequence $(m_{n})_{n\geq 1}$ is that for some constant $\const>0$ and integer $n_{0}>0$, we have 
\begin{equation}\label{eq:m}
m_{n}=\const\cdot(n+n_{0})^{1+a}\ln(n+n_{0})^{1+b}
\end{equation}
for $a>0$ and $b\geq-1$, or $a=0$ and $b>0$. 
The next assumption is essentially the same as the variance control assumption in \cite{IusJofOliTho17}.

\begin{assumption}[Variance Control]
\label{ass:variance}
For all $x\in\mathbb{R}^{d}$ and $p\geq 1$, let
\begin{equation*}
s_{p}(x)\eqdef\Ex[\norm{F(x,\xi)-T(x)}^{p}]^{1/p}.
\end{equation*}
There exist $p\geq 2,\sigma_{0}\geq 0$, and a measurable locally bounded function $\sigma:\scrX_{\ast}\to\R_{+}$ such that for all $x\in\mathbb{R}^{d}$ and all $x^{\ast}\in\scrX_{\ast}$ 
\begin{equation}\label{ineq:ass:variance}
s_{p}(x)\leq\sigma(x^{\ast})+\sigma_{0}\norm{x-x^{\ast}}. 
\end{equation}
\end{assumption}
Before we proceed with the convergence analysis, we want to make some clarifying remarks on this assumption. The most frequently used assumption on the SO's approximation error, which dates back to the seminal work of Robbins and Monro (see \cite{KusYin97,Duf96} for a textbook reference), asks for \ac{UBV}, i.e. 
\begin{equation}\label{eq:UBV}\tag{UBV}
\sup_{x\in\scrX}s_{2}(x)\leq \sigma. 
\end{equation}
\ac{UBV} is covered by \cref{ass:variance} when $\sigma_{0}=0$ and $\sup_{x\in\scrX_{\ast}}\sigma(x^{\ast})\leq \sigma$. \ac{UBV} is for instance valid when additive noise with finite $p$-th moment is assumed, that is, for some random variable $\xi\in L^{2}(\Omega,\Pr;\mathbb{R}^{d})$ with $\Ex[\norm{\xi}^{p}]^{1/p}\leq\sigma<\infty$ we have 
\begin{align*}
F(x,\xi)=T(x)+\xi\qquad\Pr\text{-a.s. }. 
\end{align*}
However, assuming a global variance bound is not realistic in cases where the variance of the stochastic oracle depends on the position $x$ (see e.g. Example 1 in \cite{JofTho18}). \cref{ass:variance} is much weaker than \ac{UBV}, as it exploits the local variance of the stochastic oracle rather than, potentially hard to estimate, global mean square variance bounds. The recent papers \cite{IusJofOliTho17,JofTho18} make similar assumptions on the variance of the stochastic oracle. It is shown there that \cref{ass:variance} is most natural in cases where the feasible set $\scrX$ is unbounded, and it is always satisfied when the Carath\'{e}odory functions $F(\cdot,\xi)$ are random Lipschitz (see \cref{ex:Lipschitz} below). Since \cref{alg:TsengFB} is an infeasible method, we are forced to analyze the behavior of the stochastic process $(X_{n},Y_{n})_{n\geq 0}$ on an unbounded domain, which makes \cref{ass:variance} the only realistic and convenient choice for us. \cref{ex:Lipschitz}  illustrates an important instance where \cref{ass:variance} holds. 
\begin{example}\label{ex:Lipschitz}
Assume for the Carath\'{e}odory mapping $F:\mathbb{R}^{d}\times\Xi\to\mathbb{R}^{d}$ that there exists $\scrL\in L^{p}(\Omega,\scrF,\Pr;\R_{+})$ with 
\begin{align*}
\norm{F(x,\xi)-F(y,\xi)}\leq \scrL(\xi)\norm{x-y}\qquad\forall x,y\in\mathbb{R}^{d}. 
\end{align*}
Call $L$ the Lipschitz constant of the map $x \mapsto T(x)=\Ex[F(x,\xi)]$. Then, a repeated application of the Minkowski inequality shows that for all $x \in \R^d$ and all $x^{\ast}\in\scrX_{\ast}$ we have
\begin{align*}
s_{p}(x)&\leq \Ex[\norm{F(x,\xi)-F(x^{\ast},\xi)}^{p}]^{1/p}+s_{p}(x^{\ast})+\norm{T(x)-T(x^{\ast})}\\
&\leq (\Ex[\scrL(\xi)^{p}]^{1/p}+L)\norm{x-x^{\ast}}+s_{p}(x^{\ast}).
\end{align*}
Let $\sigma(x^{\ast})$ denote a bound on $s_{p}(x^{\ast})$ and set $\sigma_{0}\eqdef L+\Ex[\scrL(\xi)^{p}]^{1/p}$, to get a variance bound as required in \cref{ass:variance}. 
\end{example}
\section{Convergence Analysis}
\label{sec:converge}
We consider the quadratic residual function defined by 
\begin{equation*}
r_{a}(x)^{2}\eqdef \norm{x-\Pi_{\scrX}(x-a T(x))}^{2} \ \forall x\in\mathbb{R}^{d}.
\end{equation*}
The reader familiar with the literature on finite-dimensional variational inequalities will recognize this immediately as the energy defined by the natural map $F^{\text{nat}}_{a}(x)\eqdef x-\Pi_{\scrX}(x-aT(x))$ \cite[][chapter 10]{FacPan03}. It is well known that $r_{a}(x)$ is a merit function for $\VI(T,\scrX)$. Moreover, $\{r_{a}(x);a>0\}$ is a family of equivalent merit functions for $\VI(T,\scrX)$, in the sense that $r_{b}(x)\geq r_{a}(x)$ for all $b>a>0$ \citep[][Proposition 10.3.6]{FacPan03}. 
Denote
\begin{equation}\label{eq:rho}
\rho_{n}\eqdef 1-2L^{2}\alpha_{n}^{2}\qquad\forall n\geq 0. 
\end{equation}
We define recursively the process $(V_{n})_{n\geq 0}$ by $V_{0}=0$ and, for all $n\geq 1$,
\begin{equation*}
V_{n+1}\eqdef V_{n}+(4+\rho_{n})\alpha_{n}^{2}\norm{W_{n+1}}^{2}+4\alpha_{n}^{2}\norm{Z_{n+1}}^{2},
\end{equation*}
so that 
\begin{equation}\label{eq:DeltaV}
\Delta V_{n}\eqdef V_{n+1}-V_{n}=(4+\rho_{n})\alpha_{n}^{2}\norm{W_{n+1}}^{2}+4\alpha_{n}^{2}\norm{Z_{n+1}}^{2}\qquad\forall n\geq 0. 
\end{equation}
Additionally, we define for all $x \in \R^d$ the process $(U_{n}(x))_{n\geq 0}$ given by $U_{0}(x)=0$, and 
\begin{equation*}
U_{n+1}(x)\eqdef U_{n}(x)+2\alpha_{n}\inner{Z_{n+1},x-Y_{n}}\qquad\forall n\geq 1
\end{equation*}
with corresponding increment 
\begin{equation*}\label{eq:defDeltaU}
\Delta U_{n}(x)\eqdef 2\alpha_{n}\inner{Z_{n+1},x-Y_{n}}\qquad\forall n\geq 0.
\end{equation*}
For any reference point $x\in\mathbb{R}^{d}$ we see that $\Ex[\Delta U_{n}(x)\vert\hat{\scrF}_{n}]=0$ for all $n\geq 0$. Hence, the process $\left(U_{n}(x)\right)_{n\geq 0}$ is a martingale w.r.t. the filtration $\hat{\F}$. Since $\scrF_{n}\subseteq\hat{\scrF}_{n}$, the tower property implies that 
\begin{equation}
\label{eq:DeltaU}
\Ex[\Delta U_{n}(x)\vert\scrF_{n}]=0 \qquad\forall x\in\mathbb{R}^{d} \ \forall n\geq 0,
\end{equation}
showing that it is also a $\F$-martingale.  $(V_{n})_{n\geq 0}$ is an increasing process, with increments $\Delta V_{n}$ whose expected value is determined by the variance of the approximation error of the stochastic oracle feedback. In terms of these increment processes, we establish the following fundamental recursion. 
\begin{lemma}
For all $x^{\ast}\in\scrX_{\ast}$ and all $n \geq 0$ we have 
\begin{equation}\label{eq:recursion}
\norm{X_{n+1}-x^{\ast}}^{2}\leq \norm{X_{n}-x^{\ast}}^{2}-\frac{\rho_{n}}{2}r_{\alpha_{n}}(X_{n})^{2}+\Delta U_{n}(x^{\ast})+\Delta V_{n} \quad \Pr-\mbox{a.s.}.
\end{equation}
\end{lemma}
\begin{proof}
This recursive relation follows via several simple algebraic steps. Let be $x^{\ast}\in\scrX_{\ast}$  and $n \geq 0$ fixed.
\paragraph{\underline{Step 1} } We have 
\begin{align*}
\inner{T(x^{\ast}),y-x^{\ast}}\geq 0\qquad\forall y\in\scrX. 
\end{align*}
Using that $\alpha_{n}>0$ as well as the pseudo-monotonicity of $T$, we see
\begin{align*}
\inner{\alpha_{n}T(Y_{n}),Y_{n}-x^{\ast}}\geq 0.
\end{align*}
Using the Doob decomposition in equation \eqref{eq:Doob}, we can rewrite this inequality as 
\begin{equation}
\label{eq:1}
\inner{\alpha_{n}B_{n+1},Y_{n}-x^{\ast}}\geq \alpha_{n}\inner{Z_{n+1},Y_{n}-x^{\ast}}. 
\end{equation}
Since $Y_{n}=\Pi_{\scrX}(X_{n}-\alpha_{n}A_{n+1})$, from \cref{lem:projection}(i) we conclude that 
\begin{equation}
\label{eq:2}
\inner{x^{\ast}-Y_{n},Y_{n}-X_{n}+\alpha_{n}A_{n+1}}\geq 0.
\end{equation}
Adding \eqref{eq:1} and \eqref{eq:2} gives
\begin{align*}
\inner{\alpha_{n}(A_{n+1}-B_{n+1})-X_{n}+Y_{n},x^{\ast}-Y_{n}}\geq\alpha_{n}\inner{Z_{n+1},Y_{n}-x^{\ast}},
\end{align*}
which is equivalent to 
\begin{equation}\label{eq:step1}
\inner{x^{\ast}-Y_{n},X_{n+1}-X_{n}}\geq\alpha_{n}\inner{Z_{n+1},Y_{n}-x^{\ast}}. 
\end{equation}
\paragraph{\underline{Step 2} }Using \eqref{eq:step1}, we get
\begin{align*}
\inner{X_{n+1}-X_{n},X_{n+1}-x^{\ast}} = & \inner{X_{n+1}-X_{n},Y_{n}-x^{\ast}}+\inner{X_{n+1}-X_{n},X_{n+1}-Y_{n}}\\
\leq & \inner{\alpha_{n}Z_{n+1},x^{\ast}-Y_{n}}+\norm{X_{n+1}-X_{n}}^{2}\\
& +\inner{X_{n+1}-X_{n},X_{n}-Y_{n}}\\
= & \inner{\alpha_{n}Z_{n+1},x^{\ast}-Y_{n}}+\norm{X_{n+1}-X_{n}}^{2}-\norm{X_{n}-Y_{n}}^{2}\\
&+\alpha_{n}\inner{A_{n+1}-B_{n+1},X_{n}-Y_{n}}
\end{align*}
where we have used the definition of  $X_{n+1}$ in the last equality. The Pythagoras identity in \cref{lem:Pyth} gives us
\begin{align*}
\norm{X_{n+1}-x^{\ast}}^{2}= & \norm{X_{n}-x^{\ast}}^{2}-\norm{X_{n+1}-X_{n}}^{2}+2\inner{X_{n+1}-X_{n},X_{n+1}-x^{\ast}}\\
\leq & \norm{X_{n}-x^{\ast}}^{2}+\norm{X_{n+1}-X_{n}}^{2}-2\norm{X_{n}-Y_{n}}^{2}\\
&+2\inner{\alpha_{n}Z_{n+1},x^{\ast}-Y_{n}}+2\alpha_{n}\inner{A_{n+1}-B_{n+1},X_{n}-Y_{n}}. 
\end{align*}
\paragraph{\underline{Step 3}. } Using again the definition of $X_{n+1}$, we see 
\begin{align*}
\norm{X_{n+1}-X_{n}}^{2} = & \norm{Y_{n}+\alpha_{n}(A_{n+1}-B_{n+1})-X_{n}}^{2}\\
= & \norm{X_{n}-Y_{n}}^{2}+\alpha_{n}^{2}\norm{A_{n+1}-B_{n+1}}^{2}+2\alpha_{n}\inner{A_{n+1}-B_{n+1},Y_{n}-X_{n}}\\
\leq & \norm{X_{n}-Y_{n}}^{2}+2\alpha_{n}^{2}\norm{T(X_{n})-T(Y_{n})}^{2}+2\alpha_{n}^{2}\norm{W_{n+1}-Z_{n+1}}^{2}\\
&+2\alpha_{n}\inner{A_{n+1}-B_{n+1},Y_{n}-X_{n}}\\
\leq & \norm{X_{n}-Y_{n}}^{2}+2L^{2}\alpha_{n}^{2}\norm{X_{n}-Y_{n}}^{2}+4\alpha_{n}^{2}\norm{W_{n+1}}^{2}+4\alpha_{n}^{2}\norm{Z_{n+1}}^{2}\\
&+2\alpha_{n}\inner{A_{n+1}-B_{n+1},Y_{n}-X_{n}}.
\end{align*}
The first inequality is the Cauchy-Schwarz inequality. The second inequality follows from the $L$-Lipschitz continuity of the averaged operator $T$ (\cref{ass:Lipschitz}), and again the Cauchy-Schwarz inequality. Combining this with the last inequality obtained in Step 2, we see that 
\begin{align*}
\norm{X_{n+1}-x^{\ast}}^{2}&\leq \norm{X_{n}-x^{\ast}}^{2}-(1-2L^{2}\alpha_{n}^{2})\norm{X_{n}-Y_{n}}^{2}+4\alpha_{n}^{2}\norm{W_{n+1}}^{2}\\
&+4\alpha_{n}^{2}\norm{Z_{n+1}}^{2}+2\inner{\alpha_{n}Z_{n+1},x^{\ast}-Y_{n}}. 
\end{align*}
\paragraph{\underline{Step 4} } By the definition of the squared residual function, the definition of $Y_{n}$ and \cref{lem:projection}(iii), we have 
\begin{align*}
r_{\alpha_{n}}(X_{n})^{2}&=\norm{X_{n}-\Pi_{\scrX}(X_{n}-\alpha_{n}T(X_{n}))}^{2}\\
&\leq 2\norm{X_{n}-Y_{n}}^{2}+2\norm{Y_{n}-\Pi_{\scrX}(X_{n}-\alpha_{n}T(X_{n}))}^{2}\\
&=2\norm{X_{n}-Y_{n}}^{2}+2\norm{\Pi_{\scrX}(X_{n}-\alpha_{n}A_{n+1})-\Pi_{\scrX}(X_{n}-\alpha_{n}T(X_{n})}^{2}\\
&\leq 2\norm{X_{n}-Y_{n}}^{2}+2\norm{\alpha_{n}W_{n+1}}^{2}.
\end{align*}
Hence, 
\begin{equation}
\label{eq:Step4}
-2\norm{X_{n}-Y_{n}}^{2}\leq 2\alpha_{n}^{2}\norm{W_{n+1}}^{2}-r_{\alpha_{n}}(X_{n})^{2}. 
\end{equation}
\paragraph{\underline{Step 5} } Combining \eqref{eq:Step4} with the last inequality from Step 3 and recalling \cref{ass:step}, we conclude
\begin{align*}
\norm{X_{n+1}-x^{\ast}}^{2} \leq & \norm{X_{n}-x^{\ast}}^{2}-\frac{1}{2}(1-2L^{2}\alpha_{n}^{2})r_{\alpha_{n}}(X_{n})^{2}+(1-2L^{2}\alpha_{n}^{2})\alpha_{n}^{2}\norm{W_{n+1}}^{2}\\
&+4\alpha_{n}^{2}\norm{W_{n+1}}^{2}+4\alpha_{n}\norm{Z_{n+1}}^{2}+2\inner{\alpha_{n}Z_{n+1},x^{\ast}-Y_{n}}\\
=& \norm{X_{n}-x^{\ast}}^{2}-\frac{\rho_{n}}{2}r_{\alpha_{n}}(X_{n})^{2}+(4+\rho_{n})(\alpha_n)^{2}\norm{W_{n+1}}^{2}+4\alpha_{n}^{2}\norm{Z_{n+1}}^{2}\\
&+2\inner{\alpha_{n}Z_{n+1},x^{\ast}-Y_{n}}.
\end{align*}
The definitions of the increments associated with the martingales $(U_{n}(x^{\ast}))_{n\geq 0}$ and $(V_{n})_{n\geq 0}$ give the claimed result. 
\end{proof} 
\begin{remark}
One can notice that in the above proof the pseudo-monotonicity of $T$ is used only in Step 1 of the above proof, if order to obtain relation \eqref{eq:1}. Thus, as happened in \cite{DanLan15,SolSva99}, the pseudo-monotonicity of $T$ can actually be replaced by the following weaker assumption
\begin{align*}
\inner{T(x),x-x^{\ast}}\geq 0\qquad\forall x\in\scrX,x^{\ast}\in\scrX_{\ast}.
\end{align*}
See also \cite{MerZho18} for a similar condition. 
\end{remark}
In the following, we let $p\geq 2$ be the exponent as specified in \cref{ass:variance}. Taking conditional expectations in equation \eqref{eq:recursion} and using the martingale property \eqref{eq:DeltaU}, we see for all $n\geq 0$ that
\begin{equation}
\label{eq:startFejer}
\Ex[\norm{X_{n+1}-x^{\ast}}^{2}\vert\scrF_{n}]\leq\norm{X_{n}-x^{\ast}}^{2}-\frac{\rho_{n}}{2}r_{\alpha_{n}}(X_{n})^{2}+\Ex[\Delta V_{n}\vert\scrF_{n}]. 
\end{equation}
In order to prove convergence of the process $(X_{n})_{n\geq 0}$, we aim to deduce a stochastic quasi-Fej\'{e}r relation. For that we need to understand the properties of the conditional expectation 
\[
\Ex[\Delta V_{n}\vert\scrF_{n}]=(4+\rho_{n})\alpha_{n}^{2}\Ex[\norm{W_{n+1}}^{2}\vert\scrF_{n}]+4\alpha_{n}^{2}\Ex[\norm{Z_{n+1}}^{2}\vert\scrF_{n}] \ \forall n \geq 0.  
\]
Let be $q\in[1,\infty]$. The monotonicity of $ L^{q}(\Pr) \eqdef L^{q}(\Omega,\scrF,\Pr;\R)$ norms gives $\Ex[\Delta V_{n}\vert\scrF_{n}]\leq \Ex[\abs{\Delta V_{n}}^{q}\vert\scrF_{n}]^{\frac{1}{q}}$ for all $n\geq 0$. By Minkowski inequality, 
\begin{align*}
\Ex[\abs{\Delta V_{n}}^{q}\vert\scrF_{n}]^{\frac{1}{q}}\leq (4+\rho_{n})\alpha_{n}^{2}\Ex[\norm{W_{n+1}}^{2q}\vert\scrF_{n}]^{1/q}+4\alpha_{n}^{2}\Ex[\norm{Z_{n+1}}^{2q}\vert\scrF_{n}]^{1/q} \ \forall n \geq 0.  
\end{align*}
The next lemma provides the required bounds for these expressions, and also highlights the implicit variance reduction of our method. 
\begin{lemma}\label{lem:Approx}
Let be $p' \in [2,p]$. For all $n\geq 0$ we have $\Pr$-a.s.
\begin{equation}\label{eq:boundW}
\Ex[\norm{W_{n+1}}^{p'}\vert\scrF_{n}]^{\frac{1}{p'}}\leq \frac{C_{p'}\left(\sigma(x^{\ast})+\sigma_{0}\norm{X_{n}-x^{\ast}}\right)}{\sqrt{m_{n+1}}}
\end{equation}
and 
\begin{equation}\label{eq:boundZ}
\Ex[\norm{Z_{n+1}}^{p'}\vert\scrF_{n}]^{\frac{1}{p'}}\leq \frac{C_{p'}}{\sqrt{m_{n+1}}}\left(\sigma(x^{\ast})+\sigma_{0}\Ex[\norm{Y_{n}-x^{\ast}}^{p'}\vert\scrF_{n}]^{\frac{1}{p'}}\right). 
\end{equation}
In particular, in case of \eqref{eq:UBV} with $\sigma_{0}=0$ and $\sup_{x\in\scrX_{\ast}}\sigma(x^{\ast})\leq \hat{\sigma}$, both approximation errors are bounded in $L^{p'}(\Pr)$ by the common factor $\frac{C_{p'}\hat{\sigma}}{\sqrt{m_{n+1}}}$.
\end{lemma}
\begin{proof}
See \cref{app:A}
\end{proof}
Let be $p'\geq 2$ and $n\geq 0$. We have 
\begin{align*}
\Ex[\norm{Y_{n}-x^{\ast}}^{p'}\vert\scrF_{n}]^{\frac{1}{p'}}\leq (1+\alpha_{n}L)\norm{X_{n}-x^{\ast}}+\alpha_{n}\Ex[\norm{W_{n+1}}^{p'}\vert\scrF_{n}]^{1/p'}.
\end{align*}
Hence, combining this with \eqref{eq:boundW} for $p'\in[2,p]$ as in \cref{lem:Approx}, we see that 
\begin{align*}
\Ex[\norm{Y_{n}-x^{\ast}}^{p'}\vert\scrF_{n}]^{\frac{1}{p'}}&\leq (1+\alpha_{n}L)\norm{X_{n}-x^{\ast}}+\alpha_{n}\frac{C_{p'}\left(\sigma(x^{\ast})+\sigma_{0}\norm{X_{n}-x^{\ast}}\right)}{\sqrt{m_{n+1}}}\\
&=\left(1+\alpha_{n}L+\alpha_{n}\frac{C_{p'}\sigma_{0}}{\sqrt{m_{n+1}}}\right)\norm{X_{n}-x^{\ast}}+\alpha_{n}\frac{C_{p'}\sigma(x^{\ast})}{\sqrt{m_{n+1}}}. 
\end{align*}
Plugging this inequality into  \eqref{eq:boundZ}, after rearranging the terms we see that
\begin{align*}
\Ex[\norm{Z_{n+1}}^{p'}\vert\scrF_{n}]^{\frac{1}{p'}}\leq & \frac{C_{p'}\sigma(x^{\ast})}{\sqrt{m_{n+1}}}\left(1+\alpha_{n}\frac{\sigma_{0}C_{p'}}{\sqrt{m_{n+1}}}\right)\\
&+\norm{X_{n}-x^{\ast}}\frac{C_{p'}\sigma_{0}}{\sqrt{m_{n+1}}}\left(1+\alpha_{n}L+\alpha_{n}\frac{C_{p'}\sigma_{0}}{\sqrt{m_{n+1}}}\right).
\end{align*}
We denote
\begin{equation}\label{eq:G}
G_{n,p}\eqdef\frac{C_{p}}{\sqrt{m_{n+1}}}, 
\end{equation}
such that, for all $n\geq 0$ and $p'\in[2,p]$ we obtain the expressions
\begin{align}
\label{eq:W}
\Ex[\norm{W_{n+1}}^{p'}\vert\scrF_{n}]^{\frac{1}{p'}} \leq & G_{n,p'}\left(\sigma(x^{\ast})+\sigma_{0}\norm{X_{n}-x^{\ast}}\right),\\
\label{eq:Z}
\Ex[\norm{Z_{n+1}}^{p'}\vert\scrF_{n}]^{\frac{1}{p'}} \leq  & \sigma(x^{\ast})G_{n,p'}(1+\alpha_{n}\sigma_{0}G_{n,p'})\\
\nonumber
& +\sigma_{0}G_{n,p'}\norm{X_{n}-x^{\ast}}(1+\alpha_{n}L+\alpha_{n}\sigma_{0}G_{n,p'}),\\
\label{eq:Y}
\Ex[\norm{Y_{n}-x^{\ast}}^{p'}\vert\scrF_{n}]^{\frac{1}{p'}} \leq & (1+\alpha_{n}L+\alpha_{n}\sigma_{0}G_{n,p'})\norm{X_{n}-x^{\ast}}+\alpha_{n}\sigma(x^{\ast})G_{n,p'}. 
\end{align}
In case of a \eqref{eq:UBV}, we obtain from the above estimates simple upper bounds, by setting $\sigma_{0}=0$, and replacing $\sigma(x^{\ast})$ with the uniform upper bound $\hat{\sigma}$. We next use these derived expressions to obtain $L^{q}(\Pr)$ bounds for the error increments $(\Delta U_{n}(x^{\ast}))_{n\geq 1}$ and $(\Delta V_{n})_{n\geq 1}$, when $q\in[1,p/2]$.
\begin{lemma}
Let \cref{ass:variance} be fulfilled with $p \geq 2$. For $p'\in[2,p]$, $q=\frac{p'}{2}\geq 1$ and all $n\geq 0$ we have
\begin{align}\label{eq:V}
\Ex[\abs{\Delta V_{n}}^{q}\vert\scrF_{n}]^\frac{1}{q} \leq & \ \alpha^{2}_{n}G_{n,p'}^{2}\sigma(x^{\ast})^{2}[2(4+\rho_{n})+16+16\alpha^{2}_{n}\sigma^{2}_{0}G_{n,p'}^{2}] \nonumber \\
&+\alpha^{2}_{n}G_{n,p'}^{2}\sigma^{2}_{0}\norm{X_{n}-x^{\ast}}^{2} \nonumber\\
&\quad  \times [2(4+\rho_{n})+8(1+\alpha_{n}L+\alpha_{n}\sigma_{0}G_{n,p'})^{2}]
\end{align}
and 
\begin{align}\label{eq:U}
& \Ex[\abs{\Delta U_{n}(x^{\ast})}^{q}\vert\scrF_{n}]^\frac{1}{q} \nonumber \\
\leq & \ 2\alpha^{2}_{n}G_{n,p'}^{2}\sigma(x^{\ast})^{2}(1+\alpha_{n}G_{n,p'}\sigma_{0}) \nonumber \\
& + 2\alpha_{n}G_{n,p'}\sigma(x^{\ast})\norm{X_{n}-x^{\ast}}[1+\alpha_{n}L+\alpha_{n}\sigma_{0}G_{n,p'}(3+2\alpha_{n}L)+2\alpha^{2}_{n}\sigma_{0}^{2}G^{2}_{n,p'}] \nonumber \\
& + 2\alpha_{n}G_{n,p'}\sigma_{0}\norm{X_{n}-x^{\ast}}^{2}(1+\alpha_{n}L+\alpha_{n}\sigma_{0}G_{n,p'})^{2}.
\end{align}
If \eqref{eq:UBV} holds with variance bound $\hat{\sigma}$, then these upper bounds simplify to 
\begin{align}
\label{eq:exptectedVUVB}
\Ex[\abs{\Delta V_{n}}^{q}\vert\scrF_{n}]^\frac{1}{q}&\leq \alpha^{2}_{n}\hat{\sigma}^{2}G^{2}_{n,p'}(8+\rho_{n})
\end{align}
and, respectively,
\begin{align}\label{eq:exptectedVUVB2}
\Ex[\abs{\Delta U_{n}(x^{\ast})}^{q}\vert\scrF_{n}]^\frac{1}{q}&\leq 2\alpha_{n}\hat{\sigma}G_{n,p'}(1+L\alpha_{n})\norm{X_{n}-x^{\ast}}+2\alpha^{2}_{n}\hat{\sigma}^{2}G^{2}_{n,p'}. 
\end{align}
\end{lemma}
\begin{proof} Let be $n \geq 0$. For $q\geq 1$, we know that 
\begin{align*}
\Ex[\abs{\Delta V_{n}}^{q}\vert\scrF_{n}]^\frac{1}{q}\leq (4+\rho_{n})\alpha^{2}_{n}\Ex[\norm{W_{n+1}}^{p'}\vert\scrF_{n}]^\frac{2}{p'}+4\alpha^{2}_{n}\Ex[\norm{Z_{n+1}}^{p'}\vert\scrF_{n}]^\frac{2}{p'}.
\end{align*}
Using \eqref{eq:W} and \eqref{eq:Z}, and rearranging terms, we obtain \eqref{eq:V}. On the other hand, we have by definition 
\begin{align*}
\Ex[\abs{\Delta U_{n}(x^{\ast})}^{q}\vert\hat{\scrF}_{n}]^\frac{1}{q}&\leq 2\alpha_{n}\norm{Y_{n}-x^{\ast}}\cdot \Ex[\norm{Z_{n+1}}^{q}\vert\hat{\scrF}_{n}]^\frac{1}{q}\\
&  \leq 2\alpha_{n}\norm{Y_{n}-x^{\ast}}\cdot \Ex[\norm{Z_{n+1}}^{p'}\vert\hat{\scrF}_{n}]^\frac{1}{p'}\\
 & \leq  2\alpha_{n}\norm{Y_{n}-x^{\ast}}G_{n,p'}\sigma(x^{\ast})+2\alpha_{n}G_{n,p'}\sigma_{0}\norm{Y_{n}-x^{\ast}}^{2},
 \end{align*}
where the first estimate follows from the Cauchy-Schwarz inequality, the second one uses the monotonicity of $L^{q}(\Pr)$ norms, and the third one uses eq. \eqref{eq:Z1}. Applying the operator $\Ex[\cdot\vert\scrF_{n}]$ on both sides, and using again the monotonicity of $L^{q}(\Pr)$ norms, we obtain 
\begin{align*}
\Ex[\abs{\Delta U_{n}(x^{\ast})}^{q}\vert\scrF_{n}]^\frac{1}{q} \leq & \ 2\alpha_{n}G_{n,p'}\sigma(x^{\ast})\Ex[\norm{Y_{n}-x^{\ast}}^{q}\vert\scrF_{n}]^\frac{1}{q}\\
& +2\alpha_{n}G_{n,p'}\sigma_{0}\Ex[\norm{Y_{n}-x^{\ast}}^{p'}\vert\scrF_{n}]^\frac{2}{p'}\\
\leq &  \ 2\alpha_{n}G_{n,p'}\sigma(x^{\ast})\Ex[\norm{Y_{n}-x^{\ast}}^{p'}\vert\scrF_{n}]^\frac{1}{p'}\\
&+2\alpha_{n}G_{n,p'}\sigma_{0}\Ex[\norm{Y_{n}-x^{\ast}}^{p'}\vert\scrF_{n}]^\frac{2}{p'}.
\end{align*}
After applying \eqref{eq:Y} and rearranging terms we arrive at the expression \eqref{eq:U}.

In case \ac{UBV} holds with uniform variance bound $\hat{\sigma}$, the upper bound for $\abs{\Delta V_{n+1}}^{q}$ follows immediately from the defining expression \eqref{eq:DeltaV}  by using the uniform bounds $\frac{C_{p'}\hat{\sigma}}{\sqrt{m_{n+1}}}=G_{n,p'}\hat{\sigma}$ for the quadratic error terms $\norm{W_{n+1}}^{2}$ and $\norm{Z_{n+1}}^{2}$. The corresponding bound for $\abs{\Delta U_{n}(x^{\ast})}^{q}$ is obtained from \eqref{eq:U} by setting $\sigma_{0}=0$ and replacing $\sigma(x^{\ast})$ by its uniform upper bound $\hat{\sigma}$. 
\end{proof}
Based on the previous estimates, we can now derive the announced stochastic quasi-Fej\'{e}r inequality for the sequence $\left(\norm{X_{n}-x^{\ast}}^{2}\right)_{n\geq 0}$. 
\begin{proposition}\label{prop:Fejer}
For all $x^{\ast}\in\scrX_{\ast}$ and all $n \geq 0$, we have 
\begin{equation}\label{eq:Fejer}
\Ex[\norm{X_{n+1}-x^{\ast}}^{2}\vert\scrF_{n}]\leq\norm{X_{n}-x^{\ast}}^{2}-\frac{\rho_{n}}{2}r_{\alpha_{n}}(X_{n})^{2}+\frac{\kappa_{n}}{m_{n+1}}\left[\sigma_{0}^{2}\norm{X_{n}-x^{\ast}}^{2}+\sigma(x^{\ast})^{2}\right],
\end{equation}
where 
\begin{equation*}
\kappa_{n}\eqdef \alpha^{2}_{n}C_{2}^{2}[2(4+\rho_{n})+16(1+\alpha_{n}L+\alpha_{n}\sigma_{0}G_{n,2})^{2}].
\end{equation*}
If \eqref{eq:UBV} holds with uniform variance bound $\hat{\sigma}$, then
\begin{equation}\label{eq:Fejeruniform}
\Ex[\norm{X_{n+1}-x^{\ast}}^{2}\vert\scrF_{n}]\leq\norm{X_{n}-x^{\ast}}^{2}-\frac{\rho_{n}}{2}r_{\alpha_{n}}(X_{n})^{2}+\frac{\kappa_{n}\hat{\sigma}^{2}}{m_{n+1}},
\end{equation}
where now $\kappa_{n}=\alpha^{2}_{n}C_{2}^{2}(8+\rho_{n})$. 
\end{proposition}
\begin{proof}
Let be $x^{\ast}\in\scrX_{\ast}$ and $n \geq 0$. Our point of departure is \eqref{eq:startFejer}, together with \eqref{eq:V}. From here we derive that
\begin{align*}
& \Ex[\norm{X_{n+1}-x^{\ast}}^{2}\vert\scrF_{n}]\\
\leq & \ \norm{X_{n}-x^{\ast}}^{2}-\frac{\rho_{n}}{2}r_{\alpha_{n}}(X_{n})^{2}\\
&+\alpha^{2}_{n}G_{n,2}^{2}\sigma(x^{\ast})^{2}[2(4+\rho_{n})+16+16\alpha^{2}_{n}\sigma^{2}_{0}G_{n,2}^{2}]\\
&+\alpha^{2}_{n}G_{n,2}^{2}\sigma^{2}_{0}\norm{X_{n}-x^{\ast}}^{2}[2(4+\rho_{n})+8(1+\alpha_{n}L+\alpha_{n}\sigma_{0}G_{n,p2})^{2}]\\
\leq & \ \norm{X_{n}-x^{\ast}}^{2}-\frac{\rho_{n}}{2}r_{\alpha_{n}}(X_{n})^{2}\\
&+\left(\sigma^{2}_{0}\norm{X_{n}-x^{\ast}}^{2}+\sigma(x^{\ast})^{2}\right)\left[2(4+\rho_{n})+16(1+\alpha_{n}L+\alpha_{n}\sigma_{0}G_{n,2})^{2}\right]\alpha_{n}^{2}G_{n,2}^{2}. 
\end{align*}
In the last equality, we have used that $2(4+\rho_{n})+8(1+\alpha_{n}L+\alpha_{n}\sigma_{0}G_{n,2})^{2}\leq 2(4+\rho_{n})+16(1+\alpha_{n}L+\alpha_{n}\sigma_{0}G_{n,2})^{2}$, and that $2(4+\rho_{n})+16+16\alpha_{n}^{2}\sigma_{0}^{2}G_{n,2}^{2}\leq 2(4+\rho_{n})+16(1+\alpha_{n}L+\alpha_{n}\sigma_{0}G_{n,2})^{2}$. Recalling that $G_{n,2}=C_{2}/\sqrt{m_{n+1}}$, the proof is complete.

In the case where \eqref{eq:UBV} holds, we just have to combine \eqref{eq:startFejer} with \eqref{eq:exptectedVUVB} to obtain the claimed result. 
\end{proof}

\begin{remark}\label{rem:kappa}
The scaling factor $\kappa_{n}$ only depends on the step size $\alpha_{n}$, the Lipschitz constant $L$, and the variance bound on the stochastic oracle. Let $\bar{\alpha}\eqdef\sup_{n\geq 0}\alpha_{n}$ and $\underline{\alpha}\eqdef\inf_{n\geq 0}\alpha_{n}$ (both finite and positive according to \cref{ass:step}). Using the definition of $\rho_{n}$ in \eqref{eq:rho}, we can bound
\begin{align*}
\kappa_{n}&=\alpha^{2}_{n}C_{2}^{2}\left[2(4+\rho_{n})+16(1+\alpha_{n}L+\frac{\alpha_{n}\sigma_{0}C_{2}}{\sqrt{m_{n+1}}})^{2}\right]\\
&\leq \alpha_{n}^{2}C_{2}^{2}\left[10+32(1+\alpha_{n}L)^2+32\alpha^{2}_{n}\sigma^{2}_{0}\frac{C_{2}^{2}}{m_{n+1}}\right]\\
&\leq \bar{\alpha}^{2}C_{2}^{2}\const_{1}\left[1+\frac{\bar{\alpha}^{2}\sigma_{0}^{2}C_{2}^{2}}{m_{n+1}}\right] \ \forall n \geq 0,
\end{align*}
where $\const_{1}>1$ is a constant. Combined with the batch size condition \eqref{eq:m}, we obtain the existence of constants $\const_{0}$ and $\const_{1}$ such that 
\begin{align*}
\kappa_{n}\leq \const_{1}\left(1+\frac{\bar{\alpha}^{2}\sigma_{0}^{2}C_{2}^{2}}{\const_{0}(n+n_{0})^{1+a}\ln(n+n_{0})^{1+b}}\right)
\end{align*}
for all $n\gg n_{0}$. Such non-asymptotic bounds will be used in the estimation of the rate of convergence of the algorithm. 

\end{remark}

Next we will prove that the process $(X_{n})_{n\geq 0}$ converges a.s. to a random variable $X$ with values in the set $\scrX_{\ast}$. This will be obtained as a consequence of the classical Robbins-Siegmund \cref{lem:RS}, and recent results on the convergence of stochastic quasi-F\'{e}jer monotone sequences (Proposition 2.3 in \cite{ComPes15}).

Given a stochastic process $(f_{n})_{n\geq 0} \subseteq L^{0}(\Omega,\scrF,\Pr;\mathbb{R}^{d})$, we define the (random) set of cluster points 
\begin{equation*}
\Lim(f)(\omega)\eqdef \{x\in\mathbb{R}^{d}\vert (\exists(n_{j})\uparrow\infty):\lim_{n_{j}\to\infty}f_{n_{j}}(\omega)=x\} 
\end{equation*}
\begin{theorem}
\label{th:converge}
Consider the stochastic process $(X_{n},Y_{n})_{n\geq 0}$ generated by Algorithm \ac{SFBF} under Assumptions \ref{ass:Consistent}-\ref{ass:variance}. Then, $(X_{n})_{n\geq 0}$ converges as $n \rightarrow \infty$ almost surely to a limit random variable $X$ with values in $\scrX_{\ast}$, and $\lim_{n\to\infty}\Ex[r_{\alpha_{n}}(X_{n})^{2}]=0.$ 
\end{theorem}
\begin{proof}
We fix an element $x^{\ast}\in\scrX_{\ast}$. Let $\delta_{n}(x^{\ast})\eqdef \norm{X_{n}-x^{\ast}}^{2},u_{n}\eqdef\frac{\rho_{n}}{2}r_{\alpha_{n}}(X_{n})^{2},\theta_{n} \!\eqdef \frac{\kappa_{n}\sigma_{0}^{2}}{m_{n+1}}$, and $\beta_{n}=\frac{\kappa_{n}\sigma(x^{\ast})^{2}}{m_{n+1}}$, so that \eqref{eq:Fejer} can be rewritten for all $n \geq 0$ as 
\begin{align*}
\Ex[\delta_{n+1}(x^{\ast})\vert\scrF_{n}]\leq (1+\theta_{n})\delta_{n}(x^{\ast})-u_{n}+\beta_{n}\qquad \Pr-\text{a.s.  }.
\end{align*}
Hence, by \cref{lem:RS}, there exists a random variable $\hat{\delta}(x^{\ast})\in[0,\infty)$ such that $(\delta_{n}(x^{\ast}))_{n\geq 1}\to\hat{\delta}(x^{\ast})$ a.s. as $n \rightarrow \infty$, and  $\Pr\left[\sum_{n\geq 0}u_{n}<\infty\right]=1$. In particular, $(X_{n})_{n\geq 0}$ is bounded for almost every $\omega\in\Omega$. Since $\sum_{n\geq 0}u_{n}=\sum_{n\geq 0}\rho_{n}r_{\alpha_{n}}(X_{n})^{2}\geq \hat{\rho}\sum_{n\geq 0}r_{\alpha_{n}}(X_{n})^{2}$, where $\hat{\rho}=1-2\bar{\alpha}^2L^{2}>0$, it follows that $\lim_{n\to\infty}r_{\alpha_{n}}(X_{n})=0$ $\Pr-$a.s. 

We next show that for all $\omega \in \Omega$ all limit points of $(X_{n}(\omega))_{n\geq 0}$ are points in $\scrX_{\ast}$, and then apply Proposition 2.3(iii) to conclude that $(X_{n})_{n}$ converges almost surely to a random variable $X$ with values in $\scrX_{\ast}$. Let $\omega\in\Omega$ be such that $X_{n}(\omega)$ is bounded. Since $(\alpha_{n})_{n \geq 0}$ is bounded as well, we can construct subsequences $(\alpha_{n_{j}})_{j \geq 0}$ and $(X_{n_{j}}(\omega))_{j \geq 0}$ such that $\lim_{j\to\infty}\alpha_{n_{j}}=\alpha\in[\underline{\alpha},\bar{\alpha}]$ and $\lim_{j\to\infty}X_{n_{j}}(\omega)=\chi(\omega)$. Additionally, we have $\lim_{j\to\infty}r_{\alpha_{n_{j}}}(X_{n_{j}}(\omega))=0$, so that 
\begin{align*}
\lim_{j\to\infty} X_{n_{j}}(\omega)=\lim_{j\to\infty}\Pi_{\scrX}(X_{n_{j}}(\omega)-\alpha_{n_{j}}T(X_{n_{j}}(\omega))).
\end{align*}
Therefore, by continuity of the projection operator and of the averaged map $T$, \cref{lem:projection}(iv) allows us to conclude that $\chi(\omega)\in\scrX_{\ast}$. Since the subsequence is arbitrary, it follows that $\Lim((X_{n})_{n\geq 0})(\omega)\subseteq \scrX_{\ast}$ for $\Pr$-almost all $\omega\in\Omega$. Now apply Proposition 2.3(iv) of \cite{ComPes15} to conclude that $X_{n}\to X\in L^{0}(\Omega,\scrF,\Pr;\scrX_{\ast})\; \Pr-$a.s.

To prove that $r_{\alpha_{n}}(X_{n})$ converges to $0$ in mean square as $n \rightarrow \infty$, observe first that 
\begin{align*}
\Ex[\delta_{n+1}(x^{\ast})]\leq\Ex[\delta_{n}(x^{\ast})]-\frac{\rho_{n}}{2}\Ex[r_{\alpha_{n}}(X_{n})^{2}]+\frac{\kappa_{n}}{m_{n+1}}\left(\sigma_{0}^{2}\Ex[\delta_{n}(x^{\ast})]+\sigma(x^{\ast})^{2}\right) \ \forall n \geq 0.
\end{align*}
Let $z_{n}=\Ex[\delta_{n}(x^{\ast})],u_{n}=\frac{\rho_{n}}{2}\Ex[r_{\alpha_{n}}(X_{n})^{2}]$ and $\theta_{n}$ and $\beta_{n}$ be defined as in the previous paragraph. The deterministic version of the Robbins-Siegmund \cref{lem:RS} gives $(u_{n})_{n\geq 1}\in\ell^{1}_{+}(\N)$. Hence, $\lim_{n\to\infty}\Ex[r_{\alpha_{n}}(X_{n})^{2}]=0$.
\end{proof}

\cref{th:converge} considerably strengthens similar results obtained via different splitting techniques. For \ac{SEG}, asymptotic convergence of the iterates in the sense of \cref{th:converge} is established in Theorem 3 of \cite{IusJofOliTho17}. However, different to \ac{SFBF}, \ac{SEG} requires two costly projection steps, with the same number of oracle calls. This makes Algorithm \ac{SFBF} a potentially more efficient tool, and we will demonstrate that this is actually the case empirically, as well as theoretically. Under strong monotonicity assumptions, a version of \cref{th:converge} has been recently established for a stochastic version of the classical forward-backward splitting technique in \cite{RosVilVu16}, assuming a similar variance structure on the stochastic oracle as we do. Theorem {th:converge} shows convergence of \ac{SFBF} under the much weaker assumption of  pseudo-monotonicity of the mean operator $T$. 

We close this section by reporting an improved stochastic quasi-Fej\'{e}r property in terms of the distance to the solution set $\scrX_{\ast}$. 
\begin{proposition}\label{prop:uniformFejer}
Suppose that Assumptions \ref{ass:Consistent}-\ref{ass:variance} hold. For $x^{\ast}\in\scrX_{\ast}$ set $\hat{\sigma}(x^{\ast})\eqdef \max\{\sigma(x^{\ast}),\sigma_{0}\}$, and  define $\dist(x,\scrX_{\ast})\eqdef \inf_{y\in\scrX_{\ast}}\norm{y-x}=\norm{\Pi_{\scrX_{\ast}}(x)-x}$. For all $n\geq 0$ it holds
\begin{align*}
\Ex[\dist(X_{n+1},\scrX_{\ast})^{2}\vert\scrF_{n}] \leq & \dist(X_{n},\scrX_{\ast})^{2}-\frac{\rho_{n}}{2}r_{\alpha_{n}}(X_{n})^{2}\\
& +\frac{\kappa_{n}\hat{\sigma}(\Pi_{\scrX_{\ast}}(X_{n}))^{2}}{m_{n+1}}[1+\dist(X_{n},\scrX_{\ast})^{2}]. 
\end{align*}
If \eqref{eq:UBV} holds, then we get for all $n \geq 0$ the uniform bound 
\begin{equation*}
\Ex[\dist(X_{n+1},\scrX_{\ast})^{2}\vert\scrF_{n}]\leq\dist(X_{n},\scrX_{\ast})^{2}-\frac{\rho_{n}}{2}r_{\alpha_{n}}(X_{n})^{2}+\frac{\kappa_{n}\hat{\sigma}^{2}}{m_{n+1}},
\end{equation*}
with $\kappa_{n}=\alpha^{2}_{n}C^{2}_{2}(8+\rho_{n})$. 
\end{proposition}
\begin{proof}
Let be $\pi_{n}(\omega)=\Pi_{\scrX_{\ast}}(X_{n}(\omega))$ for all $n\geq 0$ and all $\omega\in\Omega$. Since the projection operator onto the closed and convex set $\scrX_{\ast}$ is nonexpansive, we have $(\pi_{n})_{n \geq 0}\in\ell^{0}(\F)$. For all $n \geq 0$ we have
\begin{align*}
\dist(X_{n+1},\scrX_{\ast})^{2}&\leq \norm{X_{n+1}-\pi_{n}}^{2}\\
&\leq \norm{X_{n}-\pi_{n}}^{2}-\frac{\rho_{n}}{2}r_{\alpha_{n}}(X_{n})+\frac{\kappa_{n}}{m_{n+1}}[\sigma^{2}_{0}\norm{X_{n}-\pi_{n}}^{2}+\sigma(\pi_{n})^{2}]\\
&=\dist(X_{n},\scrX_{\ast})^{2}-\frac{\rho_{n}}{2}r_{\alpha_{n}}(X_{n})+\frac{\kappa_{n}\hat{\sigma}^{2}(\pi_{n})}{m_{n+1}}[\dist(X_{n},\scrX_{\ast})^{2}+1],
\end{align*}
where the second inequality uses \cref{prop:Fejer}.
\end{proof}
\section{Complexity analysis and rates}
\label{sec:complexity}
The next two propositions provide explicit norm bounds on the iterates $(X_{n})_{n\geq 0}$. These bounds are going to be crucial to assess the convergence rate and the per-iteration complexity of the proposed algorithm. To be sure, the formal appearance of the complexity estimates derived in this section is naturally similar to the corresponding bounds derived in \cite{IusJofOliTho17}. However, the key observation we would like to emphasize here is that an explicit comparison between the constants involved in the upper bounds obtained for Algorithm \ac{SFBF} with those appearing in \ac{SEG} shows that the constants are consistently smaller. This indicates that \ac{SFBF} should empirically outperform \ac{SEG}. This fact is consistently observed in all our numerical experiments, and, as we show in \cref{sec:numerics}, actually this promised gain can be quite significant. 

\begin{proposition}\label{prop:psi-bound}
Suppose that Assumptions \ref{ass:Consistent}-\ref{ass:variance} hold. For all $x^{\ast}\in\scrX_{\ast}$ let
\begin{align}\label{eq:sigma}
\hat{\sigma}(x^{\ast})&\eqdef \max\{\sigma(x^{\ast}),\sigma_{0}\},\\
\ta(x^{\ast})&\eqdef \hat{\sigma}^{2}(x^{\ast})\bar{\alpha}^{2}C_{2}^{2}\const_{1}.
\label{eq:ax}
\end{align}
Choose $n_{0}\in\N$ and $\gamma>0$ such that 
\begin{equation}\label{eq:n0}
\sum_{n\geq n_{0}}\frac{1}{m_{n+1}}\leq\gamma
\end{equation}
and 
\begin{equation}\label{eq:beta}
\beta(x^{\ast})\eqdef \gamma \ta(x^{\ast})+ \gamma^2\ta(x^{\ast})^{2} \in(0,1). 
\end{equation}
Then 
\begin{equation}\label{eq:psi-bound}
\sup_{n\geq n_{0}+1}\Ex[\norm{X_{n}-x^{\ast}}^{2}]\leq \frac{\Ex[\norm{X_{n_{0}}-x^{\ast}}^{2}]+1}{1-\beta(x^{\ast})}.
\end{equation}
\end{proposition}
\begin{proof}
We first remark that, thanks to \cref{ass:batch}, for every $\gamma>0$ we can find an index $n_{0}\in\N$ such that \eqref{eq:n0} holds. Consequently, we fix $n_{0}\in\N$ to be the smallest positive integer so that \eqref{eq:n0} holds. For all $n \geq 0$ we denote $\psi_{n}(x^{\ast})\eqdef \Ex[\norm{X_{n}-x^{\ast}}^{2}]$.  From \cref{prop:Fejer}, we obtain 
\begin{align*}
\psi_{n+1}(x^{\ast})\leq \psi_{n}(x^{\ast})-\frac{\rho_{n}}{2}\Ex[r_{\alpha_{n}}(X_{n})^{2}]+\frac{\kappa_{n}}{m_{n+1}}[\sigma_{0}^{2}\psi_{n}(x^{\ast})+\sigma(x^{\ast})^{2}] \ \forall n \geq 0.
\end{align*}
Recall from Remark \ref{rem:kappa} that 
\begin{align*}
\kappa_{n}\leq \bar{\alpha}^{2}C_{2}^{2}\const_{1}\left(1+\frac{\bar{\alpha}^{2}\sigma_{0}^{2}C_{2}^{2}}{m_{n+1}}\right)\leq\bar{\alpha}^{2}C_{2}^{2}\const_{1}\left (1+\frac{\ta(x^{\ast})}{\const_{1}m_{n+1}}\right).
\end{align*}
Using this bound, for all $n \geq n_0+1$ the previous display telescopes to 
\begin{align*}
\psi_{n}(x^{\ast})\leq\psi_{n_{0}}(x^{\ast})+\sum_{k=n_{0}}^{n-1}(1+\psi_{k}(x^{\ast}))\frac{\ta(x^{\ast})}{m_{k+1}}+\sum_{k=n_{0}}^{n-1}(1+\psi_{k}(x^{\ast}))\frac{\ta(x^{\ast})^2}{\const_{1}m^2_{k+1}}.
\end{align*}
For $p>\psi_{n_{0}}(x^{\ast})$, define $\tau_{p}(x^{\ast})\eqdef\inf\{n\geq n_{0}+1\vert \psi_{n}(x^{\ast})\geq p\} \in \mathbb{N} \cup \{+\infty\}$. We claim that there exists $\hat{p}>\psi_{n_{0}}(x^{\ast})$ such that $\tau_{\hat{p}}(x^{\ast})=\infty$. Assuming that this is not the case, then we must have that $\tau_{p}(x^{\ast})<\infty$ for all $p>\psi_{n_{0}}(x^{\ast})$. Therefore, by definition of $\tau_{p}(x^{\ast})$ and \eqref{eq:n0}, we get 
\begin{align*}
p\leq  \psi_{\tau_{p}(x^{\ast})}(x^{\ast})\leq & \ \psi_{n_{0}}(x^{\ast})+\sum_{k=n_{0}}^{\tau_{p}(x^{\ast})-1}(1+\psi_{k}(x^{\ast}))\frac{\ta(x^{\ast})}{m_{k+1}}\\
& +\sum_{k=n_{0}}^{\tau_{p}(x^{\ast})-1}(1+\psi_{k}(x^{\ast}))\frac{1}{\const_{1}}\left(\frac{\ta(x^{\ast})}{m_{k+1}}\right)^{2}\\
\leq & \ \psi_{n_{0}}(x^{\ast})+(1+p)\gamma \ta(x^{\ast})+(1+p)\frac{\gamma^2\ta(x^{\ast})^{2}}{\const_{1}}.
\end{align*}
Rearranging, and using $\const_{1}>1$ as well as \eqref{eq:beta}, gives 
\begin{align*}
p\leq \frac{\psi_{n_{0}}(x^{\ast})+1}{1-\gamma\ta(x^{\ast})-\frac{\gamma^{2}}{\const_{1}}\ta(x^{\ast})^{2}}\leq \frac{\psi_{n_{0}}(x^{\ast})+1}{1-\gamma \ta(x^{\ast}) - \gamma^2 \ta(x^{\ast})^{2}}.
\end{align*}
Since $p>\psi_{n_{0}}(x^{\ast})$ has been chosen arbitrarily, we can let $p\rightarrow\infty$ and obtain a contradiction. Therefore, there exists $\hat{p}>\psi_{n_{0}}(x^{\ast})$ such that $\bar{p}\eqdef\sup_{n\geq n_{0}+1}\psi_{n}(x^{\ast})\leq\hat{p}<\infty$. From here we get for all $n\geq n_{0}+1$
\begin{align*}
\psi_{n}(x^{\ast})&\leq\psi_{n_{0}}(x^{\ast})+\sum_{k=n_{0}}^{n-1}(1+\psi_{k}(x^{\ast}))\frac{\ta(x^{\ast})}{m_{k+1}}+\sum_{k=n_{0}}^{n-1}(1+\psi_{k}(x^{\ast}))\frac{1}{\const_{1}}\left(\frac{\ta(x^{\ast})}{m_{k+1}}\right)^{2}\\
&\leq \psi_{n_{0}}(x^{\ast})+(1+\bar{p})\gamma \ta(x^{\ast})+(1+\bar{p})\frac{\gamma^2 \ta(x^{\ast})^{2}}{\const_{1}}.
\end{align*}
Taking the supremum over $n\geq n_{0}+1$, and shifting back to the original expressions of the involved data, we get 
\[
\bar{p}=\sup_{n\geq n_{0}+1}\Ex[\norm{X_{n}-x^{\ast}}^{2}]\leq \frac{\Ex[\norm{X_{n_{0}}-x^{\ast}}^{2}]+1}{1-\beta(x^{\ast})},
\]
which further leads to \eqref{eq:psi-bound}.
\end{proof}
In case where the local variance of the \ac{SO} is uniformly bounded over the solution set $\scrX_{\ast}$, we obtain much sharper results, allowing us to bound the distance of the iterates away from the solution set. 
\begin{proposition}\label{prop:delta-bound}
Suppose that Assumptions \ref{ass:Consistent}-\ref{ass:variance} hold. Suppose the variance over the solution set $\scrX_{\ast}$ is bounded: $\hat{\sigma}(x^{\ast})\eqdef\max\{\sigma(x^{\ast}),\sigma_{0}\}\leq\hat{\sigma}$ for all $x^\ast \in\scrX_{\ast}$. Define
\begin{equation}\label{eq:a}
\ta\eqdef\bar\alpha^{2}\hat{\sigma}^{2}C_{2}^{2}\const_{1}.
\end{equation}
 Let $\phi\in(0,\frac{\sqrt{5}-1}{2})$ and choose $n_{0}\geq 1$ such that $\sum_{i\geq n_{0}}\frac{1}{m_{i+1}}\leq \frac{\phi}{\ta}$. Then 
\begin{equation}\label{eq:delta-bound}
\sup_{n\geq n_{0}+1}\Ex[\dist(X_{n},\scrX_{\ast})^{2}]\leq \frac{1+\Ex[\dist(X_{n_{0}},\scrX_{\ast})^{2}]}{1-\phi-\phi^{2}}. 
\end{equation}
\end{proposition}
\begin{proof}
We denote by $d(x)\eqdef \dist(x,\scrX_{\ast}):\mathbb{R}^{d}\to\R_{+}$  the distance function of the solution set $\scrX_{\ast}$. Since $\scrX_{\ast}$ is a nonempty, closed and convex subset of $\mathbb{R}^{d}$, the function $d(X_{n}):\Omega\to\R_{+}$ given by $\omega\mapsto d(X_{n})(\omega)\eqdef\dist(X_{n}(\omega),\scrX_{\ast})$ is $\scrF_{n}$-measurable for all $n \geq 0$. Indeed, letting $\pi_{n}(\omega)\eqdef \Pi_{\scrX_{\ast}}(X_{n}(\omega))$ for all $n \geq 0$, then first, $(\pi_{n})_{n\geq 0}\in\ell^{0}_{+}(\F)$, and second $d(X_{n})(\omega)=\norm{X_{n}(\omega)-\pi_{n}(\omega)}$ is a well-defined random process in $\ell^{0}_{+}(\F)$, being a composition of continuous and measurable functions. Therefore, for all $n \geq 0$,
\begin{align*}
\Ex[d(X_{n+1})^{2}\vert\scrF_{n}]&\leq \Ex[\norm{X_{n+1}-\pi_{n}}^{2}\vert\scrF_{n}]\\
&\leq\norm{X_{n}-\pi_{n}}^{2}-\frac{\rho_{n}}{2}\Ex[r_{\alpha_{n}}(X_{n})^{2}]+\frac{\kappa_{n}}{m_{n+1}}\left(\sigma^{2}_{0}d(X_{n})^{2}+\sigma(\pi_{n})^{2}\right).
\end{align*}
Call $\psi_{n}\eqdef \sqrt{\Ex[d(X_{n})^{2}]}$ for all $n \geq 0$. Taking expectations in the previous display, and using the assumed uniform bound of the variance, we arrive at 
\begin{align*}
\psi^{2}_{n+1}\leq \psi^{2}_{n}-\frac{\rho_{n}}{2}\Ex[r_{\alpha_{n}}(X_{n})^{2}]+\frac{\hat{\sigma}^{2}\kappa_{n}}{m_{n+1}}\left(1+\psi^{2}_{n}\right) \ \forall n \geq 0.
\end{align*}
From Remark \ref{rem:kappa}, we know that 
\begin{align*}
\kappa_{n}\leq \bar{\alpha}^{2}C_{2}^{2}\const_{1}\left(1+\frac{\bar{\alpha}^{2}\hat{\sigma}^{2}C_{2}^{2}}{m_{n+1}}\right),
\end{align*}
so that $\hat{\sigma}^{2}\kappa_{n}\leq \ta(1+\frac{\ta}{m_{n+1}\const_{1}})$ for all $n \geq 0$. Hence, for all $n \geq n_0 +1$
\begin{align*}
\psi^{2}_{n}\leq \psi^{2}_{n_{0}}+\sum_{k=n_{0}}^{n-1}(1+\psi_{k}^{2})\frac{\ta}{m_{k+1}}+\sum_{k=n_{0}}^{n-1}(1+\psi^{2}_{k})\frac{\ta^{2}}{\const_{1}m^{2}_{k+1}}.
\end{align*}
From here proceed, mutatis mutandis, as in the proof of \cref{prop:psi-bound}.
\end{proof}

We next give explicit estimates of the rate of convergence and the oracle complexity of \ac{SFBF}. The reported results are very similar to the extragradient method, with the important remark that all numerical constants can be improved under our forward-backward-forward scheme. For that purpose, it is sufficient to consider Algorithm \ac{SFBF} with a constant step size $\alpha_{n} = \alpha\in \left(0,\frac{1}{\sqrt{2}L} \right )$ for all $n \geq 0$.\footnote{The reason for this is that $\{r_{a}(x);a>0\}$ is a family of equivalent merit functions of $\VI(T,\scrX)$ (see Proposition 10.3.6 in \cite{FacPan03}, and the opening to \cref{sec:converge}). Hence, as long as the step size policy $(\alpha_{n})_{n\geq0}$ obeys \cref{ass:step}, we obtain the same rate estimates.} As in \cite{IusJofOliTho17}, we can provide non-asymptotic convergence rates for the sequence $(\Ex[r_{\alpha}(X_{n})^{2}])_{n \geq 0}$.

For all $n\geq 0, x^{\ast}\in\scrX_{\ast}$ and $\phi \in \left(0,\frac{\sqrt{5}-1}{2}\right)$, define 
\begin{align*}
&\Gamma_{n}\eqdef\sum_{i=0}^{n}\frac{1}{m_{i+1}},\; \Gamma^{2}_{n}\eqdef\sum_{i=0}^{n}\frac{1}{m_{i+1}^{2}},\\
&\rho=1-2{\alpha}^{2}L^{2},\; \delta_{n}(x^{\ast})\eqdef\norm{X_{n}-x^{\ast}}^{2}, \\
&\text{ and } H(x^{\ast},n,\phi)\eqdef \frac{1+\max_{0\leq i\leq n}\Ex[\delta_{i}(x^{\ast})]}{1-\phi-\phi^{2}}.
\end{align*} 
\begin{theorem}\label{th:rate}
Suppose that Assumptions \ref{ass:Consistent}-\ref{ass:variance} hold. Let $x^{\ast}\in\scrX_{\ast}$ be arbitrarily chosen, and consider Algorithm \ac{SFBF} with constant step size $\alpha\in \left(0,\frac{1}{\sqrt{2}L} \right )$. Choose $\phi\in \left (0,\frac{\sqrt{5}-1}{2}\right)$ and $n_{0}\eqdef n_{0}(x^{\ast})$ to be the first integer such that 
\begin{equation}\label{eq:iterate_rate}
\sum_{i\geq n_{0}}\frac{1}{m_{i+1}}\leq \frac{\phi}{\ta(x^{\ast})},
\end{equation}
where $\ta(x^{\ast})$ is defined in \eqref{eq:ax}. Let
\begin{align*}
\Lambda_{n}(x^{\ast},\phi)&\eqdef \frac{2}{\rho}\left\{\Ex[\delta_{0}(x^{\ast})]+\left(1+H(x^{\ast},n_{0},\phi)\right)\left(\ta(x^{\ast})\Gamma_{n}+\ta(x^{\ast})^{2}\Gamma_{n}^{2}\right)\right\},\\
\Lambda_{\infty}(x^{\ast},\phi)&\eqdef \sup_{n\geq 0}\Lambda_{n}(x^{\ast},\phi).
\end{align*}
For all $\eps>0$ define the stopping time 
\begin{equation}\label{eq:Neps}
N_{\eps}\eqdef \inf\{n\geq 0\vert \Ex[r_{\alpha}(X_{n})^{2}]\leq \eps\}. 
\end{equation}
Then, either $N_{\eps}=0$, or 
\begin{equation}
\label{eq:repsilon}
\Ex[r_{\alpha}(X_{N_{\eps}})^{2}]\leq\eps < \frac{\Lambda_{\infty}(x^{\ast},\phi)}{N_{\eps}}.
\end{equation}
\end{theorem}
\begin{proof}
Let $\gamma=\frac{\phi}{\ta(x^{\ast})}$, with the constant $\ta(x^{\ast})$ defined in \eqref{eq:ax}, and $n_{0}=n_{0}(x^{\ast})$ as required in the statement of the theorem. From \cref{prop:psi-bound}, we deduce the bound
\begin{align*}
\sup_{n\geq n_{0}+1}\Ex[\delta_{n}(x^{\ast})]\leq \frac{1+\Ex[\delta_{n_{0}}(x^{\ast})]}{1-\phi-\phi^{2}}\leq H(x^{\ast},n_{0},\phi). 
\end{align*}
Since $1-\phi-\phi^{2}\in(0,1)$,
$\sup_{0\leq i\leq n_{0}}\Ex[\delta_{i}(x^{\ast})]\leq H(x^{\ast},n_{0},\phi)$. Therefore, 
\begin{equation}\label{eq:boundH}
\sup_{n \geq 0} \Ex[\delta_{n}(x^{\ast})]\leq H(x^{\ast},n_{0},\phi).
\end{equation}
Taking expectations in equation \eqref{eq:Fejer}, we get 
\begin{align*}
\frac{\rho}{2}\Ex[r_{\alpha}(X_{n})^{2}]\leq \Ex[\delta_{n}(x^{\ast})]-\Ex[\delta_{n+1}(x^{\ast})]+\frac{\kappa_{n}}{m_{n+1}}\left(\sigma(x^{\ast})^{2}+\sigma_{0}^{2}\Ex[\delta_{n}(x^{\ast})]\right) \ \forall n \geq 0.
\end{align*}
Therefore, for all $n \geq 0$,
\begin{align*}
\frac{\rho}{2}\sum_{i=0}^{n}\Ex[r_{\alpha}(X_{i})^{2}]\leq\Ex[\delta_{0}(x^{\ast})]+\sum_{i=0}^{n}\frac{\kappa_{i}}{m_{i+1}}\left(\sigma(x^{\ast})^{2}+\sigma^{2}_{0}\Ex[\delta_{i}(x^{\ast})]\right). 
\end{align*}
Using the variance bound $\hat{\sigma}(x^{\ast})=\max\{\sigma(x^{\ast}),\sigma_{0}\}$, which is well defined given the local boundedness of the variance, we get first from Remark  \ref{rem:kappa} the bound 
\begin{align*}
\kappa_{i}\leq \alpha^{2}C_{2}^{2}\const_{1}\left(1+\frac{\alpha^{2}C_{2}^{2}\hat{\sigma}(x^{\ast})^{2}}{m_{i+1}}\right) \ \forall i \geq 0.
\end{align*}
Second, recalling that $\ta(x^{\ast})=\alpha^{2}\hat{\sigma}(x^{\ast})^{2}C_{2}^{2}\const_{1},$ it yields for all $n\geq0$
\begin{align*}
\frac{\rho}{2}\sum_{i=0}^{n}\Ex[r_{\alpha}(X_{i})^{2}]\leq & \ \Ex[\delta_{0}(x^{\ast})]+\sum_{i=0}^{n}\frac{\ta(x^{\ast})}{m_{i+1}}(1+\Ex[\delta_{i}(x^{\ast})])\\
&+\sum_{i=0}^{n}\frac{1}{\const_{1}}\left(\frac{\ta(x^{\ast})}{m_{i+1}}\right)^{2}(1+\Ex[\delta_{i}(x^{\ast})])\\
\leq & \ \Ex[\delta_{0}(x^{\ast})] + \left(1+\max_{0\leq i\leq n}\Ex[\delta_{i}(x^{\ast})]\right)\left(\ta(x^{\ast})\Gamma_{n}+\ta(x^{\ast})^{2}\Gamma_{n}^{2}\right).
\end{align*}
From \eqref{eq:boundH}, we conclude
\begin{align*}
\frac{\rho}{2}\sum_{i=0}^{n}\Ex[r_{\alpha}(X_{i})^{2}] \leq & \ \Ex[\delta_{0}(x^{\ast})]+(1+H(x^{\ast},n_{0},\phi))\left(\ta(x^{\ast})\Gamma_{n}+\ta(x^{\ast})^{2}\Gamma_{n}^{2}\right)\\
= & \
\frac{\rho}{2}\Lambda_{n}(x^{\ast},\phi) \ \forall n \geq 0.
\end{align*}
In conclusion,
\begin{align*}
\sum_{i=0}^{n}\Ex[r_{\alpha}(X_{i})^{2}]\leq\Lambda_{n}(x^{\ast},\phi)\qquad\forall n\geq 0.
\end{align*}
From \cref{th:converge}, we know that for all $\eps>0$ there exists $M_{\eps}\in\N$ such that $\Ex[r_{\alpha}(X_{n})^{2}]\leq \eps $ for all $n\geq M_{\eps}$. Hence, the (deterministic) stopping time $N_{\eps}$ defined in \eqref{eq:Neps} is either $0$, or an integer bounded from above. Focussing on the latter case $N_{\eps}\geq 1$, then for every $0\leq k\leq N_{\eps}-1$, we have 
\begin{align*}
\eps < \Ex[r_{\alpha}(X_{i})^{2}].
\end{align*}
From here, it follows 
\begin{align*}
\eps N_{\eps} < \sum_{i=0}^{N_{\eps}-1}\Ex[r_{\alpha}(X_{i})^{2}]\leq \Lambda_{N_{\eps-1}}(x^{\ast},\phi).
\end{align*}
Hence,
\begin{align*}
\Ex[r_{\alpha}(X_{N_{\eps}})^{2}]\leq \eps < \frac{\Lambda_{\infty}(x^{\ast},\phi)}{N_{\eps}}. 
\end{align*}
The two cases above can be compactly summarized to statement \eqref{eq:repsilon}.
\end{proof}

We next turn to the case where the local variance is uniformly bounded over the solution set. In the previous theorem, given $x^{\ast}\in\scrX_{\ast}$, the constant $\Lambda_{\infty}(x^{\ast},n_{0}(x^{\ast}),\phi)$ in the convergence rate depends on the variance and on the distance of the $n_{0}(x^{\ast})$ initial iterates to $x^{\ast}$, where $n_{0}(x^{\ast})$ and $\phi$ are chosen such that \eqref{eq:iterate_rate} holds. Assuming a uniformly bound on the variance of \ac{SO} over the solution set $\scrX_{\ast}$, we can obtain much stronger convergence rate estimates, holding uniformly over the solution set. 
\begin{proposition}\label{prop:Rate_UBV}
Assume that $\sup_{x^{\ast}\in\scrX_{\ast}}\hat{\sigma}(x^{\ast})\leq\hat{\sigma}$, where the function $\hat{\sigma}(\cdot)$ is defined in \eqref{eq:sigma}. Let $x^{\ast}\in\scrX_{\ast}$ be arbitrarily chosen, and consider Algorithm \ac{SFBF} with constant step size $\alpha\in \left(0,\frac{1}{\sqrt{2}L} \right )$. Choose $\phi\in \left (0,\frac{\sqrt{5}-1}{2}\right)$ and $n_{0}\eqdef n_{0}(\hat{\sigma})$ to be the first integer such that 
\begin{equation}\label{eq:iterationuniform}
\sum_{i\geq n_{0}}\frac{1}{m_{i+1}}\leq \frac{\phi}{\ta},
\end{equation}
where $\ta=\hat{\sigma}^{2}\alpha^{2}C^{2}_{2}\const_{1}$. Let 
\begin{align*}
&\bar{\Lambda}_{n}(\hat{\sigma}, \phi)\eqdef \frac{2}{\rho}\left\{\Ex[\dist(X_{0},\scrX_{\ast})^2]+(1+\bar H(\hat{\sigma}, n_0, \phi))(\ta\Gamma_{n}+\ta^{2}\Gamma_{n}^{2})\right\},\\
&\bar{\Lambda}_{\infty}(\hat{\sigma}, \phi)=\sup_{n\geq 0}\bar{\Lambda}_{n}(\phi,\hat{\sigma}), \text{and }\\
&\bar H(\hat{\sigma}, n_0, \phi))\eqdef\frac{1+\max_{0\leq i\leq n_{0}(\hat{\sigma})}\Ex[\dist(X_{i},\scrX_{\ast})]}{1-\phi-\phi^{2}}.
\end{align*}
For all $\eps>0$ consider the stopping time defined in \eqref{eq:Neps}. Then, either $N_{\eps}=0$, or 
\begin{equation}\label{eq:rateuniform}
\Ex[r_{\alpha}(X_{N_{\eps}})^{2}]\leq \eps < \frac{\bar{\Lambda}_{\infty}(\hat{\sigma}, \phi)}{N_{\eps}},
\end{equation}
\end{proposition}
\begin{proof}
The proof is almost identical to the proof of \cref{th:rate}, but now we will use the estimates from \cref{prop:uniformFejer} and \cref{prop:delta-bound} . We first remark that the upper variance bound $\hat{\sigma}$ is the only parameter in this statement; hence, the threshold index $n_{0}=n_{0}(\hat{\sigma})$ depends on this parameter only. Once we made this choice, we can repeat all the steps involved in the proof of \cref{th:rate} verbatim, but by using \cref{prop:uniformFejer} instead of \cref{prop:Fejer}, to conclude that
\begin{align*}
\sum_{i=0}^{n}\Ex[r_{\alpha}(X_{i})^{2}]\leq & \ \Ex[\dist(X_{0},\scrX_{\ast})^{2}]+\ta\sum_{i=0}^{n}\frac{1+\Ex[\dist(X_{i},\scrX_{\ast})^{2}]}{m_{i+1}}\\
&+\ta^{2}\sum_{i=0}^{n}\frac{1+\Ex[\dist(X_{i},\scrX_{\ast})^{2}]}{m_{i+1}^{2}} \ \forall n \geq 0.
\end{align*}
\cref{prop:delta-bound} gives us 
\begin{align*}
\sup_{n\geq n_{0}+1}\Ex[\dist(X_{n},\scrX_{\ast})^{2}]\leq \frac{1+\Ex[\delta(X_{n_{0}},\scrX_{\ast})^{2}]}{1-\phi-\phi^{2}}\leq \bar H(\hat{\sigma}, n_0, \phi),
\end{align*}
from which it follows 
\begin{align*}
\sup_{n\geq 0}\Ex[\dist(X_{n},\scrX_{\ast})^{2}]\leq \bar H(\hat{\sigma}, n_0, \phi). 
\end{align*}
From here, we conclude just as in the proof of \cref{th:rate} that 
\begin{align*}
\sum_{i=0}^{n}\Ex[r_{\alpha}(X_{i})^{2}]&\leq\bar{\Lambda}_{n}(\hat{\sigma}, \phi)\leq\bar{\Lambda}_{\infty}(\hat{\sigma}, \phi) \ \forall n \geq 0.  
\end{align*}
Choose $\eps>0$ arbitrary, and consider the stopping time \eqref{eq:Neps}. Then, either $N_{\eps}=0$, or else $N_{\eps}\geq 1$. Focussing on the latter case, we argue just as in the proof of \cref{th:rate}, that 
\begin{align*}
\eps N_{\eps} < \sum_{i=0}^{N_{\eps}-1}\Ex[r_{\alpha}(X_{i})^{2}]\leq\Lambda_{N_{\eps}-1}(x^{\ast},\phi,\hat{\sigma}).
\end{align*}
Hence, if $N_\eps$ not zero, we must have
\begin{align*}
\Ex[r_{\alpha}(X_{N_{\eps}})^{2}]\leq\eps < \frac{\Lambda_{\infty}(\hat{\sigma}, \phi)}{N_{\eps}}.
\end{align*}

\end{proof}
We now turn to the estimate of the oracle complexity. By this we mean the overall size of the data set needed to be processed in order to make the natural residual function smaller than a given tolerance level $\eps >0$, in mean square. Hence, using the stopping time \eqref{eq:Neps}, we would like to estimate the number $\sum_{i=0}^{N_{\eps}}2m_{i+1}$. 

For simplicity, we will assume that the local variance function $\sigma(x^{\ast})$ is uniformly bounded over the solution set $\scrX_{\ast}$. That is, we assume that there exists $\hat{\sigma}\in(0,\infty)$ such that $\sup_{x\in\scrX_{\ast}}\hat{\sigma}(x)\leq\hat{\sigma}$. A more complete argument, without making this strong assumption can be given similar to Proposition 3.23 in \cite{IusJofOliTho17}.  We refrain doing so, since our main aim in this paper is to illustrate the improvement in the convergence rate when using Algorithm \ac{SFBF} instead of \ac{SEG}, and the simplest setting is enough for this purpose. We organize the derivation of an oracle complexity estimate in two parts. First, we will show that a specific (though admissible) choice of the sample rate, allows us to give an explicit bound on the number of preliminary iterates $n_{0} \eqdef n_0(\hat \sigma)$ needed to apply the general bounds reported in \cref{prop:Rate_UBV}. Building on this insight, we directly estimate the oracle complexity. 

As announced, we first establish a bound on the number of iterations we need to meet condition \eqref{eq:iterationuniform}.
\begin{lemma}
Let $\ta$ be the constant defined in \eqref{eq:a}, and $\phi\in(0,\frac{\sqrt{5}-1}{2})$. We choose the sample rate 
\begin{equation}\label{eq:batchuniform}
m_{i}=\lceil \theta(\mu-1+i)\ln(\mu+i-1)^{1+b}\rceil,
\end{equation}
for $i\geq 1,\theta>0,\mu>1$ and $b>0$. Then, if $n_{0}$ is an integer satisfying 
\begin{align*}
n_{0}\geq 1-\mu+e^{\left(\frac{\ta}{\phi\theta b}\right)^{1/b}},
\end{align*}
we have $\sum_{i\geq n_{0}}\frac{1}{m_{i+1}}\leq \frac{\phi}{\ta}$.
\end{lemma}
\begin{proof}
For $n_{0}\geq 1$, we compute 
\begin{align*}
\sum_{i\geq n_{0}}\frac{1}{m_{i+1}}&\leq \frac{1}{\theta}\sum_{i\geq n_{0}}\frac{1}{(i+\mu)\ln(i+\mu)^{1+b}}\\
&\leq \frac{1}{\theta}\int_{n_{0}-1}^{\infty}\frac{1}{(t+\mu)\ln(t+\mu)^{1+b}}\dif t\\
&=\frac{1}{\theta b\ln(n_{0}-1+\mu)^{b}}. 
\end{align*}
Therefore, if $\frac{1}{\theta b\ln(n_{0}-1+\mu)^{b}}\leq\frac{\phi}{\ta}$, we obtain the desired bound. Solving the latter inequality for $n_{0}$ gives the claimed result. 
\end{proof}

Using the sample rate \eqref{eq:batchuniform}, we will now bound the constant $\bar{\Lambda}(\hat{\sigma}, \phi)$, and the stopping time $N_{\eps}$. Define the constants 
\begin{align*}
\scrA_{\mu,b}\eqdef \frac{\alpha^{2}C^{2}_{2}\const_{1}}{b\ln(\mu-1)^{b}},\; \scrB_{\mu,b}\eqdef \frac{\alpha^{4}C_{2}^{4}\const_{1}^{2}}{(1+2b)(\mu-1)\ln(\mu-1)^{1+2b}}. 
\end{align*}
Since, 
\[
\Gamma_{\infty}\leq \frac{1}{\theta b}\frac{1}{\ln(\mu-1)^{b}},\text{ and }\Gamma^{2}_{\infty}\leq\frac{1}{\theta^{2}}\frac{1}{(2b+1)(\mu-1)\ln(\mu-1)^{1+2b}},
\]
 we conclude 
\begin{align*}
\ta\Gamma_{\infty}+\ta^{2}\Gamma^{2}_{\infty}&\leq\max\{1,\theta^{-2}\}(\scrA_{\mu,b}\hat{\sigma}^{2}+\scrB_{\mu,b}\hat{\sigma}^{4}).
\end{align*}
Therefore, 
\begin{align*}
\bar{\Lambda}(\hat{\sigma},\phi)&\leq \max\{1,\theta^{-2}\}\left\{ \frac{2}{\rho}\Ex[\dist(X_{0},\scrX_{\ast})^{2}]+\frac{2}{\rho}(1+\bar H(\hat{\sigma}, n_0, \phi))\left[\scrA_{\mu,b}\hat{\sigma}^{2}+\scrB_{\mu,b}\hat{\sigma}^{4}\right]\right\}\\
&\eqdef \max\{1,\theta^{-2}\}\scrQ(\phi,\hat{\sigma}).
\end{align*}
This yields the following refined uniform bound on the squared residual function.
\begin{corollary} 
For all $\eps>0$, the stopping time $N_{\eps}$ defined in \eqref{eq:Neps} is either zero, or
\begin{align*}
\Ex[r_{\alpha}(X_{N_{\eps}})^{2}]\leq\eps < \frac{\max\{1,\theta^{-2}\}\scrQ(\phi,\hat{\sigma})}{N_{\eps}}.
\end{align*}
\end{corollary}
We now turn to the estimation of the oracle complexity.  To this end, we have to bound the total number of data points involved in the $N_{\eps}$ batches needed to execute Algorithm \ac{SFBF}, i.e. we want to upper bound the sum $2\sum_{i=0}^{N_{\eps}}m_{i}$. Given the definition of the sample rate in \eqref{eq:batchuniform}, we can perform the following computation: 
\begin{align*}
\sum_{i=1}^{N_{\eps}+1}m_{i}&\leq \max\{1,\theta\}\left[\ln(N_{\eps}+\mu+1)^{1+b}\sum_{i=1}^{N_{\eps}+1}(i-1+\mu)+(N_{\eps}+1)\right] \\
&\leq \max\{1,\theta\}\left[\ln(N_{\eps}+1+\mu)^{1+b}\frac{(N_{\eps}+1)}{2}(N_{\eps}+2\mu)+(N_{\eps}+1)\right].
\end{align*}
Hence, 
\begin{equation}\label{eq:mIntermediate}
2\sum_{i=1}^{N_{\eps}}m_{i}\leq\max\{1,\theta\}(N_{\eps}+1)(N_{\eps}+2\mu)\left[\ln(N_{\eps}+1+\mu)^{1+b}+\frac{2}{N_{\eps}+2\mu}\right].
\end{equation}
\begin{proposition}
Let $\eps\in(0,1)$ be arbitrarily chosen, and $\mu\in(1,1/\eps)$. Define 
\begin{align*}
\scrI(\hat{\sigma}, \phi)\eqdef & 3\left(\frac{2}{\rho}\Ex[\dist(X_{0},\scrX_{\ast})^{2}]+2\right)^2\\
&+\frac{12}{\rho^{2}}(1+\bar H(\hat{\sigma}, n_0, \phi))^{2}\scrA_{\mu,b}^{2}\hat{\sigma}^{4}+\frac{12}{\rho^{2}}(1+\bar H(\hat{\sigma}, n_0, \phi))^{2}\scrB_{\mu,b}^{2}\hat{\sigma}^{8},\\
\scrJ(\hat{\sigma}, \phi)\eqdef & \bar{\Lambda}_{\infty}(\hat{\sigma}, n_0, \phi)+2.
\end{align*}
If the sample rate $(m_{i})_{i\geq 1}$ is given by \eqref{eq:batchuniform}, then we can bound the oracle complexity by 
\[
2\sum_{i=1}^{N_{\eps}+1}m_{i}\leq \frac{2\max\{1,\theta\}\max\{1,\theta^{-4}\}\scrI(\hat{\sigma}, \phi)\left(\ln(\scrJ(\hat{\sigma}, \phi)/\eps)^{1+b}+\mu^{-1}\right)}{\eps^{2}}.
\]
\end{proposition}
\begin{proof}
The proof is patterned after \cite{IusJofOliTho17}. Using $N_{\eps} < \bar{\Lambda}_{\infty}(\phi,\hat{\sigma})/\eps$, we continue from \eqref{eq:mIntermediate}  to obtain the bound 
\begin{align*}
2\sum_{i=1}^{N_{\eps}+1}m_{i}&\leq\max\{1,\theta\}\frac{(\bar{\Lambda}_{\infty}(\hat{\sigma, \phi})+1)(\bar{\Lambda}_{\infty}(\hat{\sigma}, \phi)+2)}{\eps^{2}}\left[\ln\left(\frac{\bar{\Lambda}_{\infty}(\hat{\sigma}, \phi)+2}{\eps}\right)^{1+b}+\mu^{-1}\right]\\
&\leq \max\{1,\theta\}\frac{(\bar{\Lambda}_{\infty}(\hat{\sigma}, \phi)+2)^{2}}{\eps^{2}}\left[\ln(\eps^{-1}\scrJ(\hat{\sigma}, \phi))^{1+b}+\mu^{-1}\right].
\end{align*}
Since, 
\begin{align*}
& (\Lambda_{\infty}(\hat{\sigma}, \phi)+2)^{2}\\
\leq & \max\{1,\theta^{-4}\}\left\{ \frac{2}{\rho}\Ex[\dist(X_{0},\scrX_{\ast})^{2}]+\frac{2}{\rho}(1+\bar H(\hat{\sigma}, n_0, \phi))\left[\scrA_{\mu,b}\hat{\sigma}^{2}+\scrB_{\mu,b}\hat{\sigma}^{4}\right]+2\right\}^{2}\\
\leq & \max\{1,\theta^{-4}\}3\left(\frac{2}{\rho}\Ex[\dist(X_{0},\scrX_{\ast})^{2}]+2\right)^{2}+\frac{12}{\rho^{2}}(1+\bar H(\hat{\sigma}, n_0, \phi))^{2}\scrA_{\mu,b}^{2}\hat{\sigma}^{4}\\
&+\frac{12}{\rho^{2}}\max\{1,\theta^{-4}\}(1+\bar H(\hat{\sigma}, n_0, \phi))^{2}\scrB_{\mu,b}^{2}\hat{\sigma}^{8}\\
=& \max\{1,\theta^{-4}\}\scrI(\hat{\sigma}, \phi),
\end{align*}
the result follows.
\end{proof}

\section{Computational Experiments}
\label{sec:numerics}
We provide four examples to verify our theoretical results and compare our methods with the \ac{SEG} proposed in \cite{IusJofOliTho17}. All experiments, beside \cref{exp:EE}, were generated with Matlab R2017a on a Linux OS with a 2.39 Ghz processor and 16 GB of memory. \cref{exp:EE} was generated with Mathematica 11 on a MacBook Pro with a 2.9 Ghz processor and 16 GB memory.

\subsection{Fractional programming and applications to communication networks}
\label{sec:fracprog}

Due to its widespread use and applications, fractional programming is instrumental to operations research and engineering, ranging from network science to signal processing, wireless communications and many other related fields \cite{SheWei2018a}. The standard form of a stochastic fractional program is as follows:
\begin{equation}
\label{eq:frac}
\begin{aligned}
\textrm{minimize}
	&\quad
	f(x)
		= \Ex\bigg[ \frac{G(x;\xi)}{h(x;\xi)} \bigg],
	\\
\textrm{subject to}
	&\quad
	x\in\scrX
\end{aligned}
\end{equation}
where $G$ and $h$ are positive and convex in $x$ for all $\xi$. It is well known that such problems are pseudo-convex \cite{BV04}, so they fall within the general framework of this paper. In particular, one of the cases most commonly encountered in practice is when $h$ is linear in $x$ and deterministic, i.e.,
\[
h(x;\xi)
	\eqdef h(x)
	= a^{\top} x + b
\]
for vectors $a$ and $b$ of suitable dimension. Solving this problem directly involves the pseudo-monotone operator $T(x)=\nabla f(x)$. Indeed, $x^{\ast}\in\scrX$ solves problem \eqref{eq:frac} if and only if $x^{\ast}$ solves $\VI(T,\scrX)$. 
\begin{experiment}[Quadratic fractional programming]
\label{exp:quadratic}

In our first experiment, we consider functions $G$ of the form
\[
G(x,\xi)
	= \frac{1}{2}x^{\top}Q(\xi)x
	+ c(\xi)^{\top}x
	+ q(\xi), 
\]
where $Q(\xi)\in\R^{d\times d}$, $c(\xi)\in\R^d$ and $q(\xi)\in\R$ are randomly generated, and $Q$ is further assumed to be positive semi-definite.
More specifically, the problem data for $Q$ is randomly generated as follows:
\[
Q= M^{\top}M + \Id,
\]
where $M$ is a random matrix of size $d\times d$ and $\Id$ is the $d\times d$ identity matrix. Finally, the vectors $a$ and $c$ are drawn uniformly at random from $(0,2)^{d}$, $q$ is a random number in $(1,2)$, and $b=1+4d$.

At each sample of the methods, we generate a sample matrix as
\[
Q(\xi)
	= Q + \frac{1}{2}\left( V(\xi)+V(\xi)^{T} \right),
\]
where $V(\xi)$ is a $d\times d$ random matrix with iid entries drawn from a normal distribution with zero mean and standard derivation $\sigma=0.1$. Similarly,
\begin{equation}
c(\xi):=c + c_1(\xi), \quad q(\xi) = q+q_1(\xi),
\end{equation}
where $c_1(\xi)$ and $q(\xi)$ are a random vector and a random number with zero mean and normal distribution with derivation $\sigma=0.1$, respectively. Also, for the problem's feasible region, we consider box constraints of the form
\begin{equation}
\scrX	= \{x \in \mathbb{R}^d: a_{i} \leq x_i \leq b_{i} \quad i=1, ..., d\},
\end{equation}
where the lower bound $a_{i}$ is a random vector in $(0,1)^d$ and the upper bound $b_{i}=a_{i}+10$. We have implemented \ac{SEG} and \ac{SFBF} for this problem, using the random operator $F(x,\xi)=\nabla_{x}\left(\frac{G(x,\xi)}{h(x)}\right)$. The starting point $x_0$ is randomly chosen in $(1,10)^d$. Both algorithms are run with a constant step-size policy. We fix the stepsize of \ac{SFBF} and \ac{SEG} as $\alpha_{FBF} = 10/d$ and $\alpha_{EG} = \alpha_{FBF}/\sqrt{3}$. The step-size $\alpha_{EG}$ is the largest one compatible with the theory developed in \cite{IusJofOliTho17}. We choose the batch size sequence $m_{n+1}=\left[ \frac{(n+1)^{1.5}}{d}\right]$, so that Assumption 6 is satisfied. We stop the algorithms when the residual is below a given tolerance $\eps$. Specifically, our stopping criterion is 
\[
r_n\eqdef\norm{ x_n-\Pi_\scrX(x_n - T(x_n))}\leq \eps =10^{-3}.
\]
Our numerical experiments involve dimension $d\in\{200,500,1000,2000\}$, and for each value of $d$ we perform $10$ runs and compare the average number of iterations and CPU time. The results are displayed in \cref{TableFraProg}, \cref{FracProgBox1,FracProgBox2}. It can be seen that \ac{SFBF} is constantly about $1.5$ faster than \ac{SEG} in both computational time and number of iterations. An interesting observation is that the number of iterations seems not to depend on the problem dimension.

\begin{table}
	\caption{Averaged over 100 runs for fractional problems of different size }\label{TableFraProg}
	\bigskip
	\centering
	\renewcommand{\arraystretch}{1.25}
	\begin{tabular}{|c | c c | c c| }
		\hline
		~&\ac{SFBF}&&\ac{SEG}&~\\
		d &number of iterations&time(sec.)&number of iterations&time(sec.)\\
		\hline
	200 &29.88&0.0473&43.96&0.0835\\
		\hline
	500 &29.84&0.2647&44.49&0.3793\\
		\hline
	1000 &30.14&1.1650&44.99&1.7017\\
		\hline
  2000 & 30.54 &8.0487&45.68&11.4803\\
	  \hline	
	\end{tabular}
\end{table}

 \begin{figure}[ht!]
 	\centering
 	\includegraphics[width=0.45\textwidth]{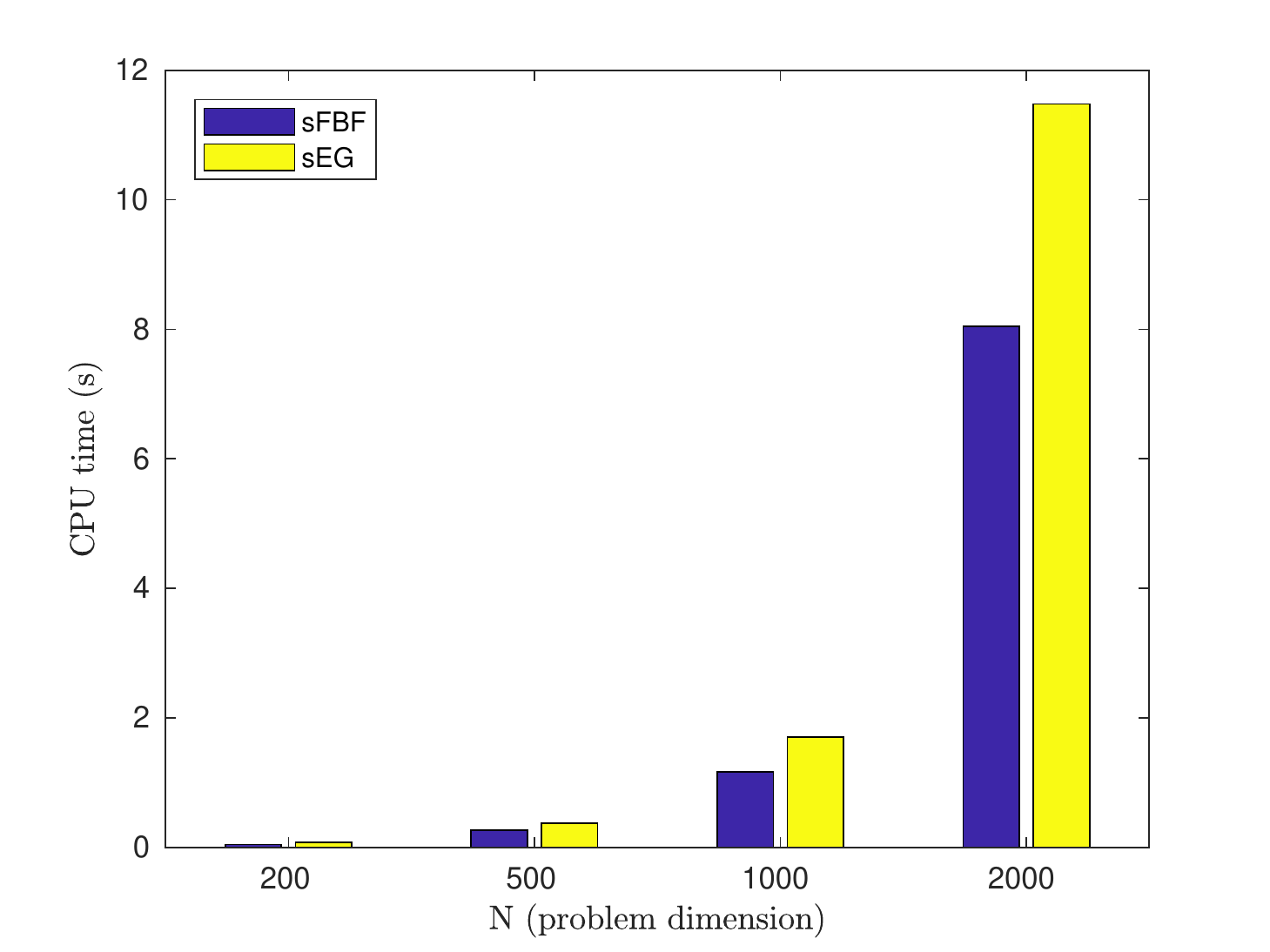}
 	\includegraphics[width=0.45\textwidth]{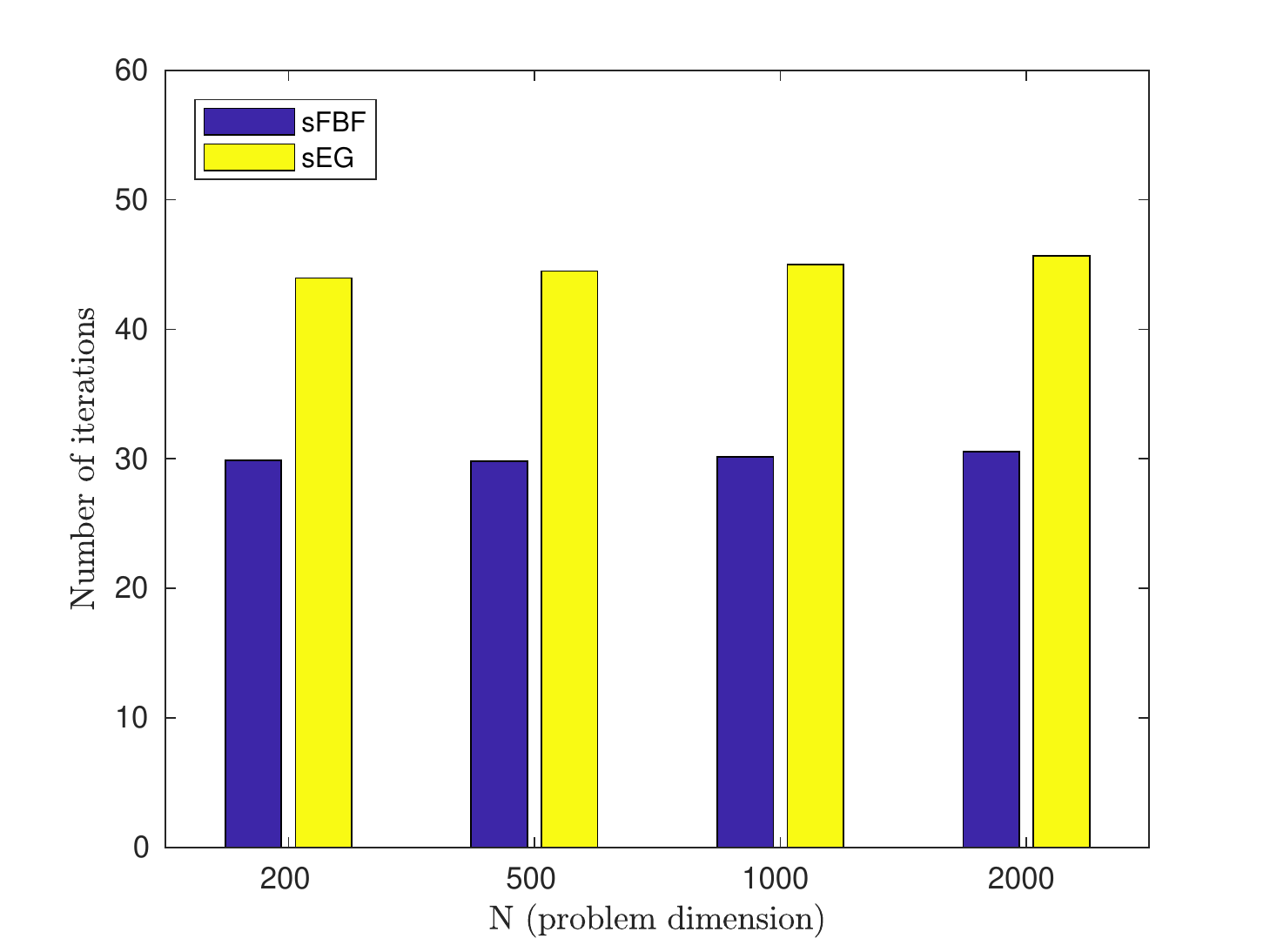}
 	\caption{Comparison between \ac{SFBF} and \ac{SEG} for solving the fractional programming.
 		We represent the averaged CPU time (left) and averaged number of iterations (right) for 100 random instances.}\label{FracProgBox1}
 \end{figure} 
 
 \begin{figure}[ht!]
 	\centering
  \includegraphics [width=0.45\textwidth]{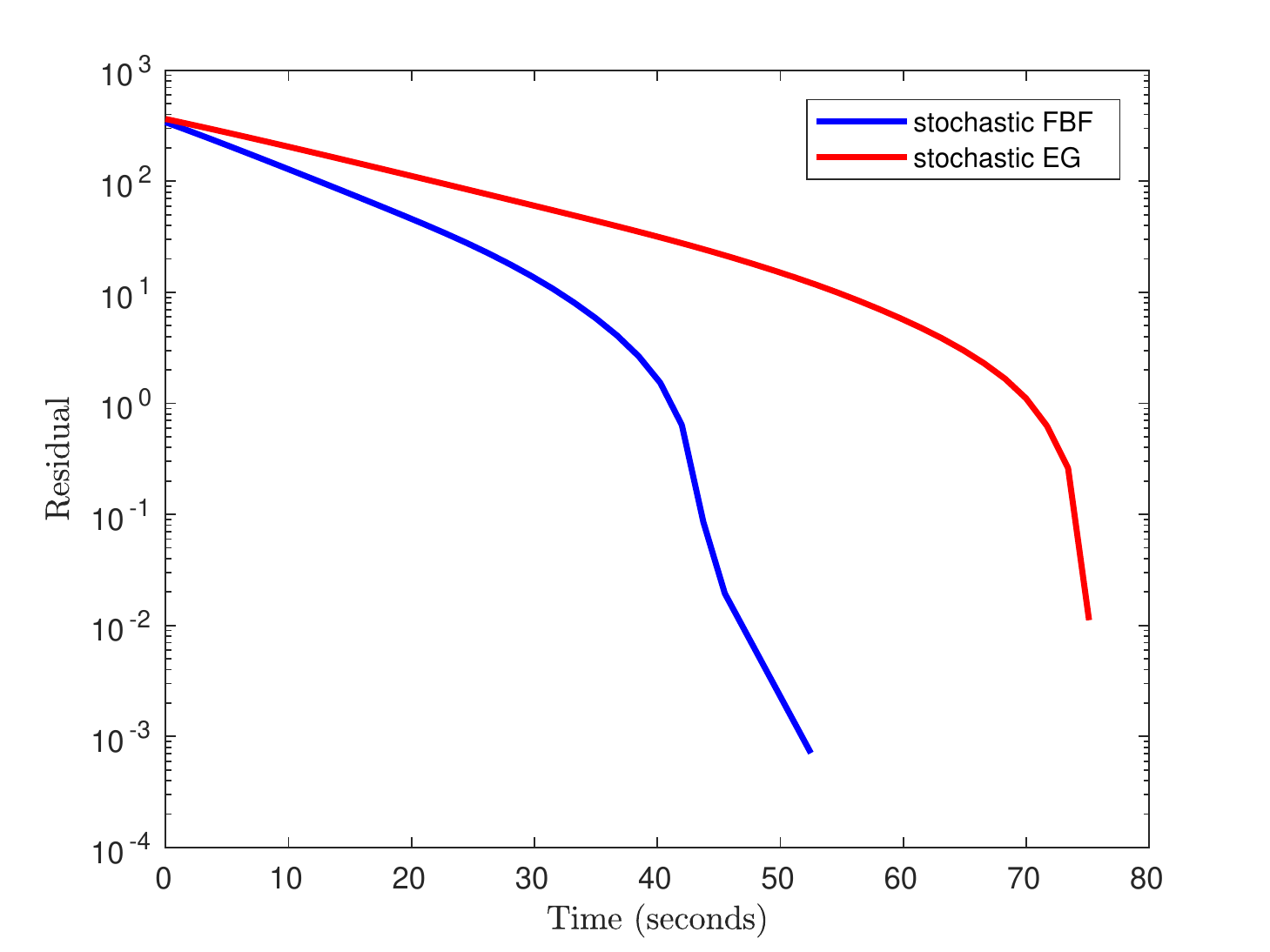} 
 	\includegraphics [width=0.45\textwidth]{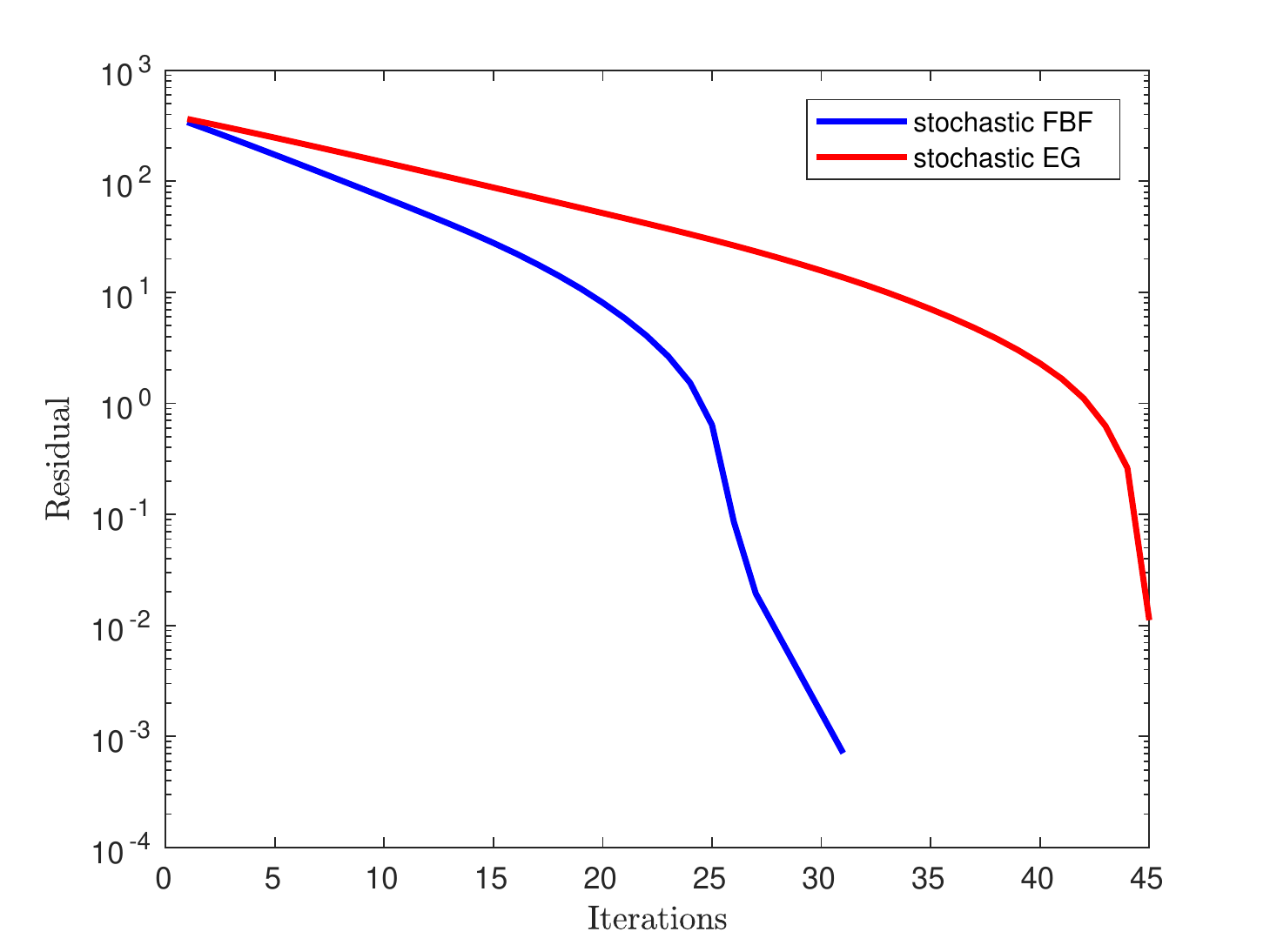}   		
 	\caption{Comparison between \ac{SFBF} and \ac{SEG} for solving the fractional programming.
 		We represent the residual vs.
running time (left) and number of iterations (right) for one random example $n=5000$.}\label{FracProgBox2}
\end{figure} 

\end{experiment}

\newcommand{\EE}{\operatorname{EE}}
\newcommand{\rate}{R}

\newcommand{\tx}{M}
\newcommand{\rx}{N}

\newcommand{\bH}{H}
\newcommand{\bI}{I}
\newcommand{\bX}{X}
\newcommand{\bz}{z}

\newcommand{\effH}{\tilde H}
\newcommand{\pc}{P^{c}}
\newcommand{\pt}{P^{t}}
\newcommand{\pmax}{P_{\max}}

\begin{experiment}[Energy efficiency in multi-antenna communications]
\label{exp:EE}
Energy efficiency is one of the most important requirements for mobile systems, and it plays a crucial role in preserving battery life and reducing the carbon footprint of multi-antenna devices (i.e., wireless devices equipped with several antennas to multiplex and demultiplex received or transmitted signals).

Following \cite{ICJF12,FJLC+13,MB16}, the problem can be formulated as follows:
consider $K$ wireless devices (e.g., mobile phones), each equipped with $\tx$ transmit antennas and seeking to connect to a common base-station with $\rx$ receiver antennas.
In this case, the users' achievable throughput (received bits/sec) is given by the familiar Shannon\textendash Telatar capacity formula \cite{Tel99}:
\begin{equation}
\label{eq:rate}
\textstyle
\rate(X;H)
	= \log\det\left( \Id + \sum_{k=1}^{K} H_{k} X_{k} H_{k}^{\dag} \right),
\end{equation}
where:
\begin{enumerate}
\addtolength{\itemsep}{\smallskipamount}
\item
$X_{k}$ 
is the $\tx\times\tx$ Hermitian \emph{input signal covariance matrix} of user $k$ and $X = (X_{1},\dotsc,X_{K})$ denotes their aggregate covariance profile.
As a covariance matrix, each $X_{k}$ is Hermitian positive semi-definite.
\item
$H_{k}$ 
is the $\rx\times\tx$ \emph{channel matrix} of user $k$, representing the quality of the wireless medium between user $k$ and the receiver.
\item
$\Id$ is the $\rx\times\rx$ identity matrix.
\end{enumerate}
\medskip

In practice, because of fading and other signal attenuation factors, the channel matrices $H_{k}$ are random variables, so the users' achievable throughput is given by
\begin{equation}
\label{eq:rate-ex}
\rate(X)
	= \Ex_{H}[\rate(X;H)],
\end{equation}
where the expectation is taken over the (often unknown) law of $H$.
The system's \acdef{EE} is then defined as the ratio of the users' achievable throughput per the unit of power consumed to achieved, i.e.,
\begin{equation}
\label{eq:EE}
\EE(X)
	= \frac{\rate(X)}{\sum_{k=1}^{K} [\pc_{k} + \pt_{k}]},
\end{equation}
where
\begin{enumerate}
\addtolength{\itemsep}{\smallskipamount}
\item
$\pt_{k}$ is the transmit power of the $k$-th device;
by elementary signal processing considerations, it is given by $\pt_{k} = \tr(X_{k})$.
\item
$\pc_{k} > 0$ is a constant representing the total power dissipated in all circuit components of the $k$-th device (mixer, frequency synthesizer, digital-to-analog converter, etc.), \emph{except} for transmission.
For concision, we will also write $\pc = \sum_{k} \pc_{k}$ for the total circuit power dissipitated by the system.
\end{enumerate}
\smallskip
The users' transmit power is further constrained by the maximum output of the transmitting device, corresponding to a trace constraint of the form
\begin{equation}
\tr(X_{k})
	\leq \pmax
	\quad
	\forall k=1,\dotsc, K.
\end{equation}
Hence, putting all this together, we obtain the stochastic fractional problem:
\smallskip
\begin{equation}
\label{eq:EE-opt}
\begin{aligned}
\textrm{maximize}
	&\quad
	\EE(X)
		= \frac{\Ex_{H}[\rate(X;H)]}{\pc + \sum_{k=1}^{K} \tr(X_{k})}
	\\[.5ex]
\textrm{subject to}
	&\quad
	X_{k} \mgeq 0,
	\\
	&\quad
	\tr(X_{k}) \leq \pmax\qquad \forall k=1,2,\ldots,K.
\end{aligned}
\end{equation}
\smallskip
Note that the overall problem dimension is $d=KM^{2}$. The \acl{EE} objective of this problem (which, formally, has units of bits/Joule) has been widely studied in the literature \cite{CGB04,ICJF12} and it captures the fundamental trade-off between higher spectral efficiency and increased battery life.
Importantly, switching from maximization to minimization, we also see that \eqref{eq:EE-opt} is of the general form \eqref{eq:frac}, so it can be solved by applying the \ac{SFBF} algorithm:
in fact, given the costly projection step to the problem's feasible region, \ac{SFBF} seems ideally suited to the task.

We do so in a series of numerical experiments reported in \cref{fig:EE}.
Specifically, we consider a network consisting of $K=16$ users, each with $\tx = 4$ transmit antennas, and a common receiver with $\rx = 128$ receive antennas.
To simulate realistic network conditions, the users' channel matrices are drawn at each update cycle from a COST Hata radio propagation model with Rayleigh fading \cite{Hat80};
to establish a baseline, we also ran an experiment with static, deterministic channels.
For comparison purposes, we ran both \ac{SFBF} and \ac{SEG} with the same variance reduction schedule, the same number of iterations, and step-sizes chosen as in \cref{exp:quadratic}; also, to reduce statistical error, we performed $S=100$ sample runs for each algorithm.
As in the case of \cref{exp:quadratic}, the \ac{SFBF} algorithm performs consistently better than \ac{SEG}, converging to a given target value between $1.5$ and $3$ times faster.


\begin{figure}[tbp]
\centering
\footnotesize
\includegraphics[width=.49\textwidth]{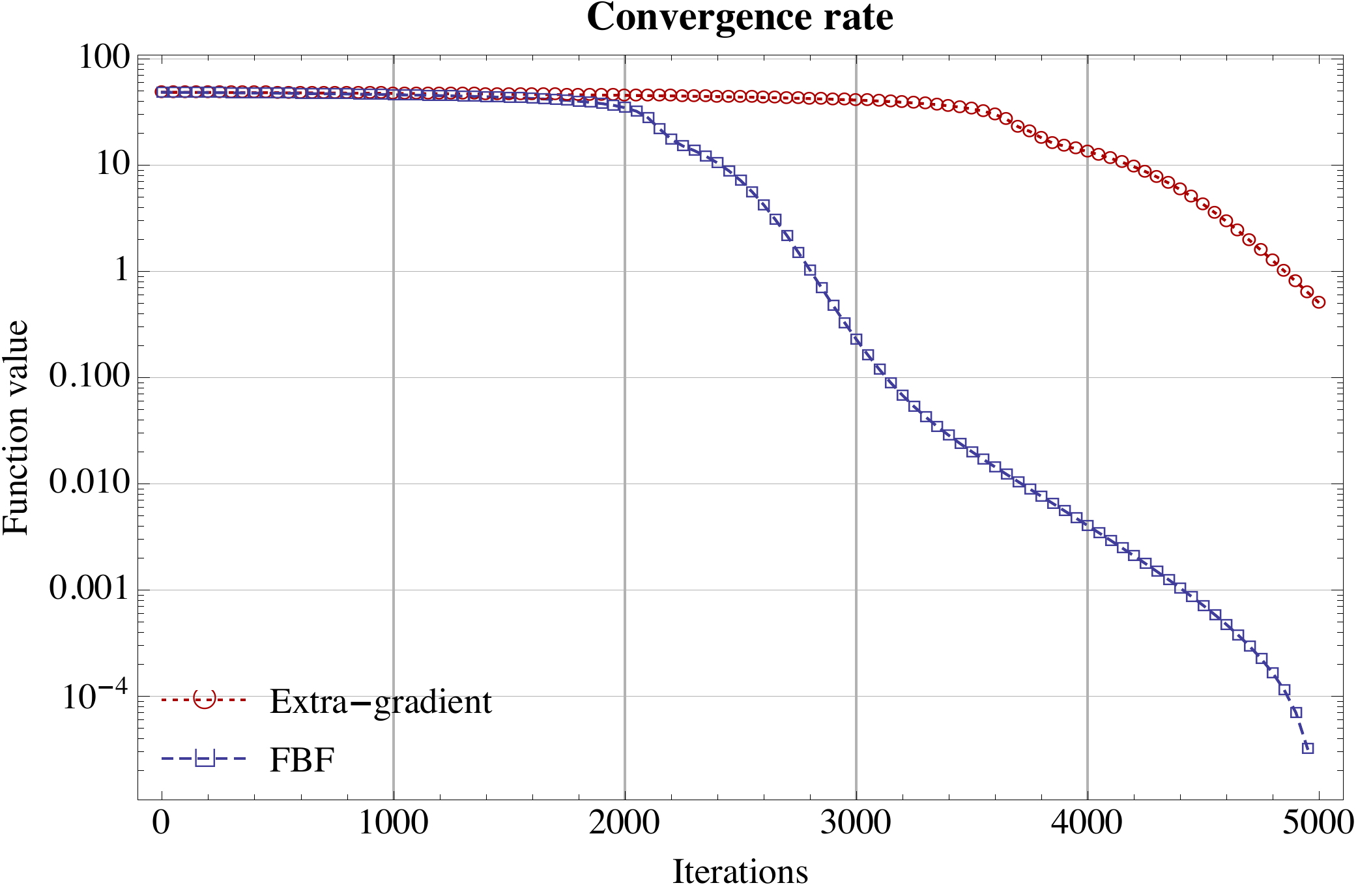}
\hfill
\includegraphics[width=.475\textwidth]{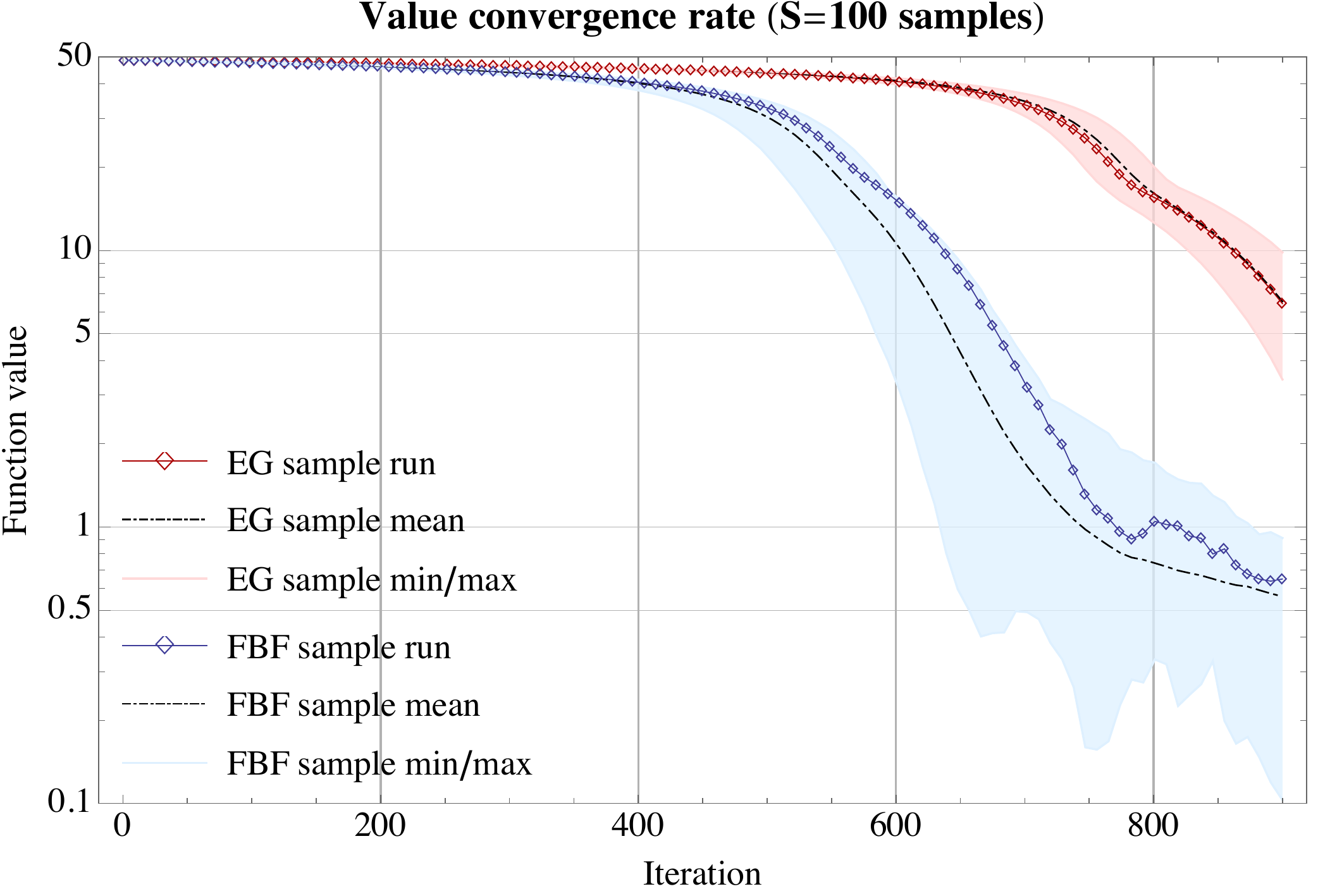}
\caption{Comparison of the extra-gradient and \ac{FBF} methods in the energy efficiency maximization problem \eqref{eq:EE-opt}.
On the left, we considered static channels, and we ran \ac{SFBF} and \ac{SEG} with the same initialization.
On the right, we considered ergodic channels following a Rayleigh fading model and we performed $S = 100$ sample runs for each algorithm;
we then plotted a sample run, the sample mean, and the best and worst values at each iteration for each algorithm.
In all cases, \ac{SFBF} exhibits significant performance gains over \ac{SEG}.}
\label{fig:EE}
\end{figure}


\end{experiment}

\subsection{Matrix Games}
As numerical illustration we investigate the performance of the algorithm to compute Nash equilibria in random matrix games. To be specific, we revisit in this experiment the problem of computing one Nash equilibrium in random two-player bimatrix games. A bimatrix game presented in its mixed extension consists of a tuple $\scrG=\left(\{I,II\},(u_{I},u_{II}),(S_{I},S_{II})\right)$, defined by
\begin{itemize}
	\item the set of players $\{I,II\}$;
	\item strategy sets $S_{I}\eqdef\{p\in\R^{n_{I}}_{+}\vert \sum_{i=1}^{n_{I}}p_{i}=1\},S_{II}\eqdef\{q\in\R^{n_{II}}_{+}\vert \sum_{i=1}^{n_{II}}q_{i}=1\}$;
	\item real valued utility functions $u_{I}(p,y)\eqdef p^{\top}U_{I}q,u_{II}(p,q)\eqdef q^{\top}U_{II}^{\top}y$, defined by
	the matrices $(U_{I},U_{II})$, both of which are real matrices of dimension $n_{I}\times n_{II}$.
\end{itemize}
Recall that a pair of mixed actions $(p^{\ast},q^{\ast})$ is called a Nash equilibrium of the bimatrix game $(U_{I},U_{II})$, if
\begin{align*}
&p_{i}^{\ast}>0\Rightarrow (U_{I}q)_{i}=\max_{1\leq j\leq n_{I}}(U_{I}q)_{j}\text{ and }\\
&q_{i}^{\ast}>0\Rightarrow (U_{II}^{\top}p)_{i}=\max_{1\leq j\leq n_{II}}(U_{II}^{\top}p)_{j}.
\end{align*}
The bimatrix game $\scrG$ is symmetric if $n_{I}=n_{II}$ and $U_{I}=U_{II}$. In symmetric games, it is natural to focus on symmetric Nash equilibria, which is a Nash equilibrium $(p^{\ast},q^{\ast})$ with $p^{\ast}=q^{\ast}$.

Let $d\eqdef n_{I}+n_{II}$, and note that $\R^{d}\cong \R^{n_{I}}\times \R^{n_{II}}$, via the usual embedding of a pair $(p,q)$ to a stacked vector in $\R^{d}$. Define the $d\times d$ matrix
\begin{equation}\label{eq:M}
M\eqdef\left[\begin{array}{cc} 0 & -U_{I}\\ -U_{II}^{\top} & 0\end{array}\right],
\end{equation}
and consider the set
\begin{equation}
\scrX\eqdef \{(x_{1},x_{2})\in\R^{n_{I}}_{+}\times\R^{n_{II}}_{+}\vert U_{I}x_{2}\leq\1_{n_{I}}\text{ and }U_{II}^{\top}x_{1}\leq\1_{n_{II}}\}.
\end{equation}
It is a classical fact that a Nash equilibrium $(p^{\ast},q^{\ast})$ can be computed by finding a pair $(x_{1},x_{2})\neq(\0_{n_{I}},\0_{n_{II}})\in\scrX$ such that
\begin{align*}
x_{1}^{\top}(\1_{n_{I}}-U_{I}x_{2})=0,\text{ and }x_{2}^{\top}(\1_{n_{II}}-U_{II}^{\top}x_{1})=0.
\end{align*}
The payoffs of the players in equilibrium can be recovered by looking at $v=\frac{1}{\sum_{j=1}^{n_{I}} x_{1,j}},u=\frac{1}{\sum_{i=1}^{n_{II}}x_{2,i}}$, and the mixed actions defining equilibrium play are recovered by $p=x_{1}\cdot v,q=x_{2}\cdot u$. It is clear that $(\0_{n_{I}},\0_{n_{II}})$ is always a solution to the \emph{linear complementarity problem}
\begin{equation}\label{eq:LCP}
\left\{\begin{array}{l}
x_{1}^{\top}(\1_{n_{I}}-U_{I}x_{2})=0, \1_{n_{I}}-U_{I}x_{2}\geq \0_{n_{I}},\\
x_{2}^{\top}(\1_{n_{II}}-U_{II}^{\top}x_{1})=0,\1_{n_{II}}-U_{II}^{\top}x_{1}\geq \0_{n_{II}}.
\end{array}\right.
\end{equation}
This the so-called \emph{artificial equilibrium} of the game, and serves as the initial point in the most used algorithm for computing Nash equilibria in bimatrix games, the Lemke-Howson algorithm, as masterly surveyed in \cite{VonStengel02}. Defining the mapping $T:\R^{d}\cong \R^{n_{I}}\times \R^{n_{II}} \to \R^{d}\cong \R^{n_{I}}\times \R^{n_{II}}$, by 
\begin{equation} \label{eq:Tpart}
T(x)\eqdef \left[\begin{array}{c} \1_{n_{I}}\\ \1_{n_{II}}\end{array}\right]+Mx
\end{equation}
we can reformulate the conditions \eqref{eq:LCP} compactly as
\begin{equation}
x^{\ast}\geq\0_{n}\text{ and }T(x^{\ast})\geq\0_{n},\inner{x^{\ast},T(x^{\ast})}=0.
\end{equation}

To turn this into a stochastic complementarity problem, we consider a stochastic Nash game \cite{KanShan12,DuvMerStaVer18}, where the player set and the set of mixed actions if fixed, but the payoff functions are realizations of random matrices 
\begin{align*}
U_{I}^{n}=U_{I}(\xi_{n}),U_{II}^{n}=U_{II}(\xi_{n})
\end{align*}
and $(\xi_{n})$ is a random process in some set $\Xi$, defined on a probability space $(\Omega,\scrF,\Pr)$. For each $n\geq 1$, we look at that random operator 
\begin{equation}
F(x,\xi_{n})\eqdef \left[\begin{array}{c} \1_{n_{I}}\\ \1_{n_{II}}\end{array}\right]+M(\xi_{n})x,
\end{equation}
and run Algorithm \ac{SFBF}. 


In our experiments, $M$ is defined as in \eqref{eq:M} and $d=n_{I}+n_{II}$.  Each element of the matrices $U_{I}, U_{II}$ is generated randomly with uniform distribution in $(0,1)$. To setup the experiments, we generate random matrices 
$M(\xi) := M+ V(\xi)$, where $V(\xi) $ is a $d\times d$ random matrix with zero mean and normal distribution with derivation $\sigma=0.1$. Since the operator $T$ is Lipschitz continuous with modulus $L=\|M\|$, we run \ac{SEG} and \ac{SFBF} with constant stepsizes $\alpha_{FBF}= \frac{0.99}{\sqrt{2}L}$, and $\alpha_{EG}= \frac{0.99}{\sqrt{6}L}$, respectively. We choose the batch size sequence $m_{n+1}=\left[ \frac{(n+1)^{1.5}}{d}\right] $ so that \cref{ass:batch} is satisfied. The same stopping criterion as in the previous experiments of \cref{sec:fracprog} is used. 

From the numerical experiments, we observe that the \ac{SFBF} outperforms the \ac{SEG}, being on average 1.7 times faster in computational time and 1.5 times faster in number of iterations. The difference becomes larger as the problem dimension increases. There are two reasons for results: firstly, \ac{SEG} requires two projections per iteration while \ac{SFBF} only requires one and more importantly, the stepsize of \ac{SFBF} is $\sqrt{3}$ times larger than that of \ac{SEG}.  

\begin{experiment}[Zero-Sum games]
\label{ex:zerosum}
We compare the performance  \ac{SFBF} and \ac{SEG} for zero sum game, i.e.,  $U_{I}=-U_{II}^{T}$. The results are displayed in \cref{ZeroSum} and \cref{Barchart1} showing the advantage of \ac{SFBF} over \ac{SEG}. On average, \ac{SFBF} is 1.7 times faster in computational time and 3.4 times faster in number of iterations than \ac{SEG}. 

\begin{table}
	\caption{Averaged over 100 runs for zero sum game of different size }\label{ZeroSum}
	\bigskip
	\centering
	\renewcommand{\arraystretch}{1.25}
	\begin{tabular}{|c | c c | c c| }
		\hline
		Dimension &\ac{SFBF} &&\ac{SEG}&~\\
		$d=n_I+n_{II}$ &Iterations&time(sec.)&Iterations&time(sec.)\\
		\hline
		$n_I=n_{II}=100$ &84.38&0.4421&172.42&1.4768\\
		\hline
		$n_I=n_{II}=250$ &214.09&9.2088&372.80&32.4321\\
		\hline
		$n_I=n_{II}=500$ &430.18&73.9068&749.65&270.5911\\
		\hline
		$n_I=n_{II}=1000$ &865.67  & 672.0806& 1508.50&2535.50\\
		\hline	
	\end{tabular}
\end{table}

\end{experiment}

\begin{experiment}[Symmetric game]
\label{exp:sym}
 We compare the performance \ac{SFBF} and \ac{SEG} for symmetric game, i.e., $U_{I},U_{II}$ are symmetric and  $U_{I}=U_{II}^{T}$. We choose $n_I=n_{II}\in\left\lbrace 50, 100, 150,\ldots,500 \right\rbrace $ and $d=n_{I}+n_{II}$. The results are displayed in \cref{SymmetricGame} and \cref{Barchart1} showing the advantage of \ac{SFBF} over \ac{SEG}. 

\begin{table}
	\caption{Averaged over 100 runs for symmetric game of different size }\label{SymmetricGame}
	\bigskip
	\centering
	\renewcommand{\arraystretch}{1.25}
	\begin{tabular}{|c | c c | c c| }
		\hline
		Dimension &\ac{SFBF} && \ac{SEG} &~\\
		$d=n_I+n_{II}$ &Iterations&time(sec.)&Iterations&time(sec.)\\
		\hline
		$n_I=n_{II}=100$ &52.00&0.3882&68.68&0.6293\\
		\hline
		$n_I=n_{II}=250$ &97.96&2.589&142.55&5.1276\\
		\hline
		$n_I=n_{II}=500$ &173.30&10.5297&247.30&21.0797\\
		\hline
		$n_I=n_{II}=1000$ & 319.92 &92.0417&455.48&191.6854\\
		\hline	
	\end{tabular}
\end{table}

 \begin{figure}[ht!]
 	\centering
 	\includegraphics[width=0.45\textwidth]{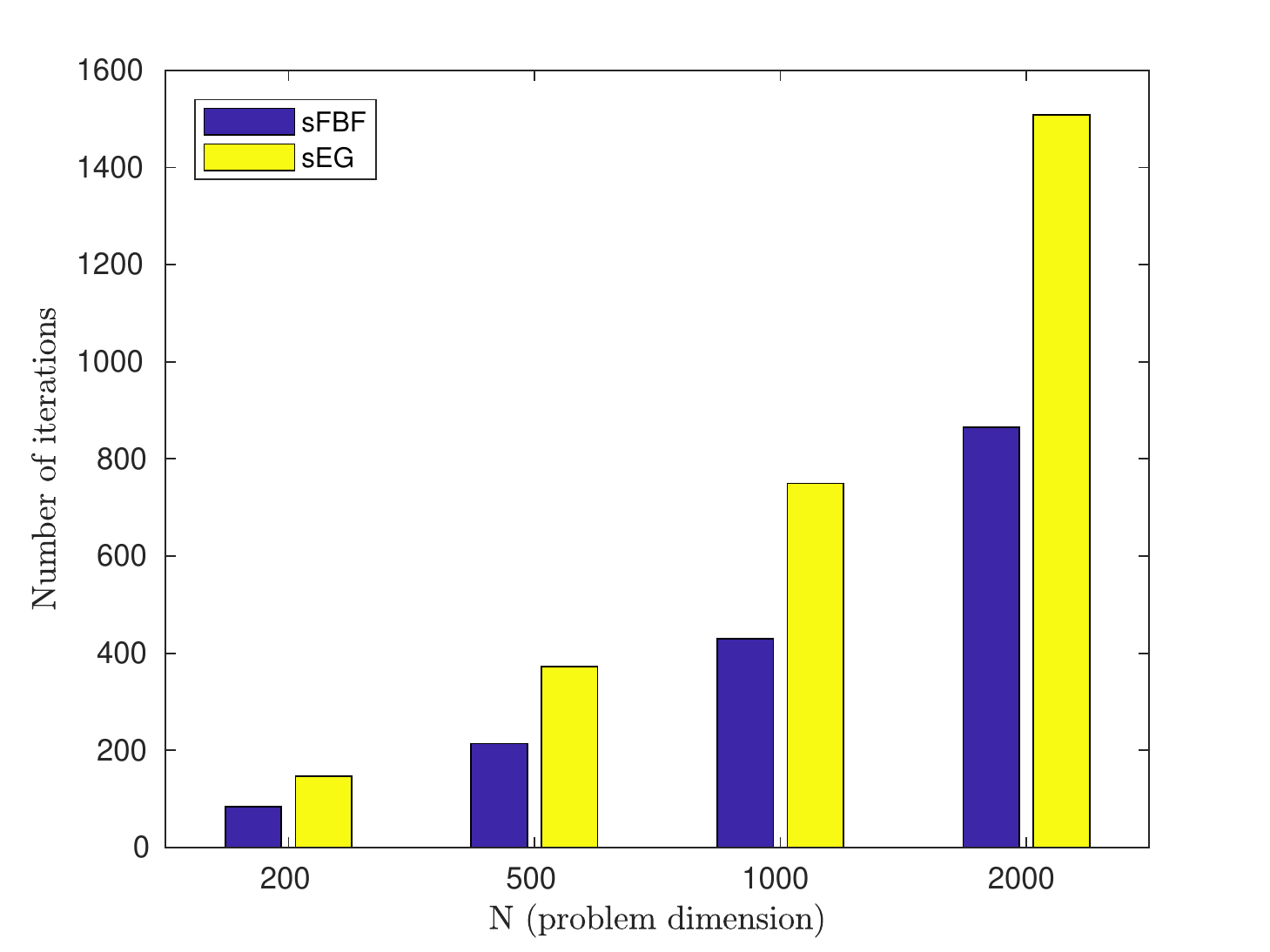}
 	\includegraphics[width=0.45\textwidth]{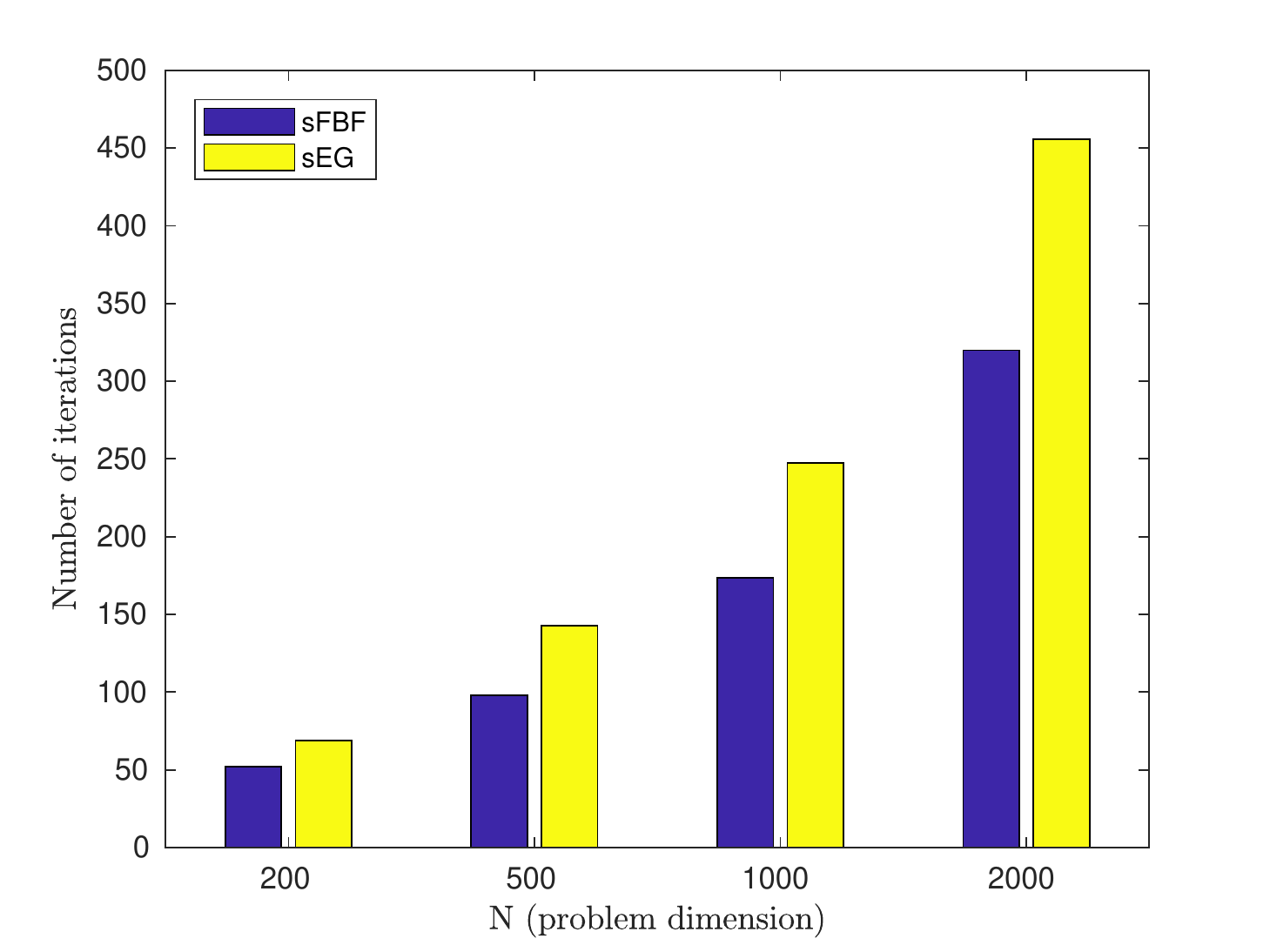}
 	\caption{Comparison on number of iterations  between \ac{SFBF} and \ac{SEG} for solving zeros sum game (left) and symmetric game (right).
 		}\label{Barchart1}
 \end{figure} 
 
 \end{experiment}
 
\begin{experiment}[Bimatrix Games]
\label{exp:bimatrix}
We compare the performance  SFBF and SEG  for asymmetric game. We choose $n_I\in\left\lbrace 100, 200,\ldots,1000 \right\rbrace $ and $n_{II}=2 n_I$. The results are displayed in \cref{AsymmetricGame} and \cref{Asymmetric_Fig1}  and \cref{Asymmetric_Fig2} showing the advantage of SFBF over SEG. 

\begin{table}
	\caption{Averaged over 100 runs for asymmetric game of different size }\label{AsymmetricGame}
	\bigskip
	\centering
	\renewcommand{\arraystretch}{1.25}
	\begin{tabular}{|c | c c | c c| }
		\hline
		~&\ac{SFBF} && \ac{SEG} &~\\
		$d=n_I+n_{II}$ &Iterations&time(sec.)&Iterations&time(sec.)\\
		\hline
		$n_I=100, n_{II}=200$ &100.28&1.9553&155.28&4.8202\\
		\hline
		$n_I=300, n_{II}=600$ &293.36&32.3010&466.01&90.2339\\
		\hline
		$n_I=500, n_{II}=1000$ &492.21&136.7019&779.86&394.7606\\
		\hline
		$n_I=1000, n_{II}=2000$ &992.64&1597.7266&1564.12 &46559.2133\\
		\hline	
	\end{tabular}
\end{table}

  \begin{figure}[ht!]
  	\centering
  	\includegraphics [width=0.45\textwidth]{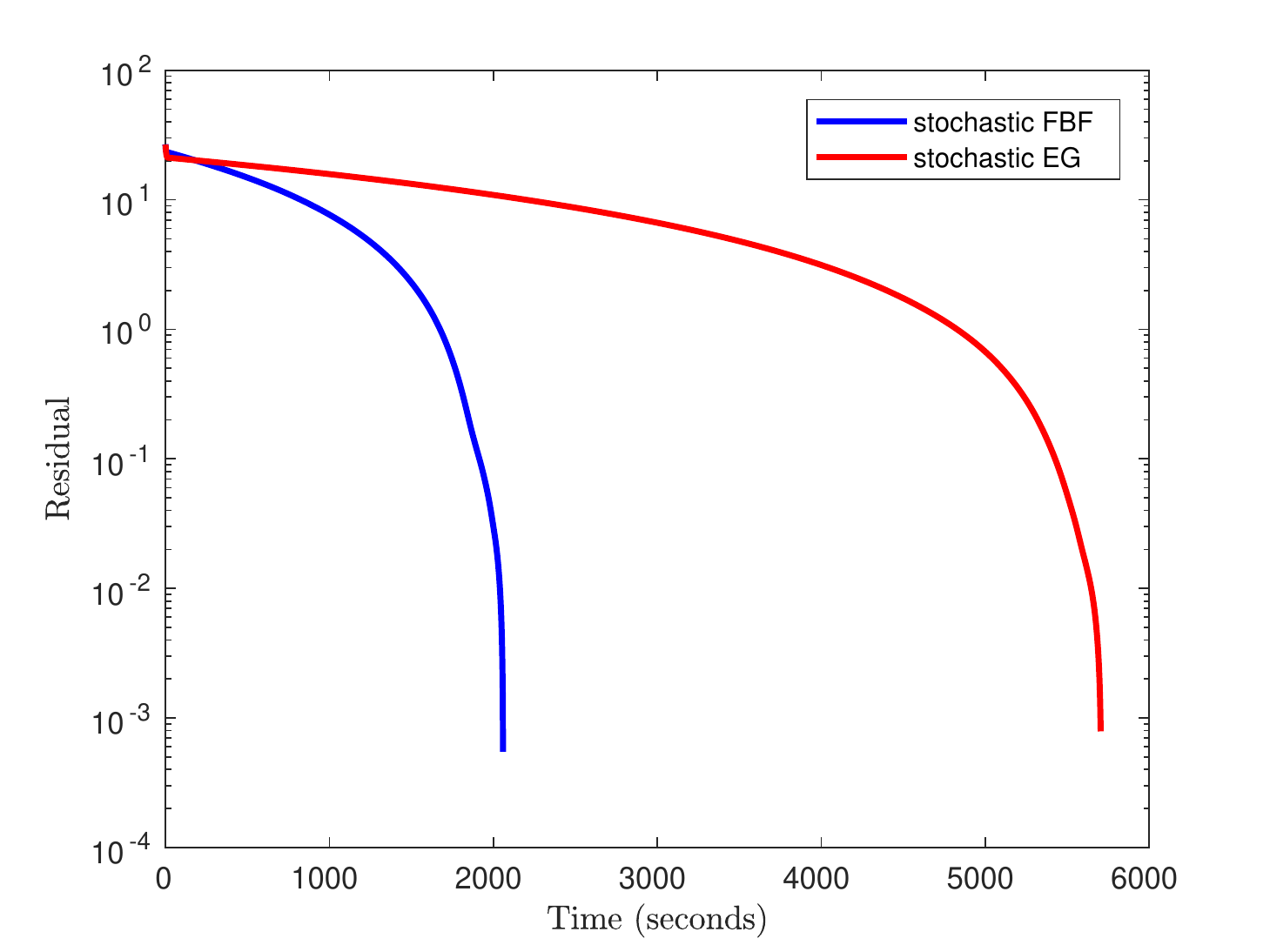} 
  	\includegraphics [width=0.45\textwidth]{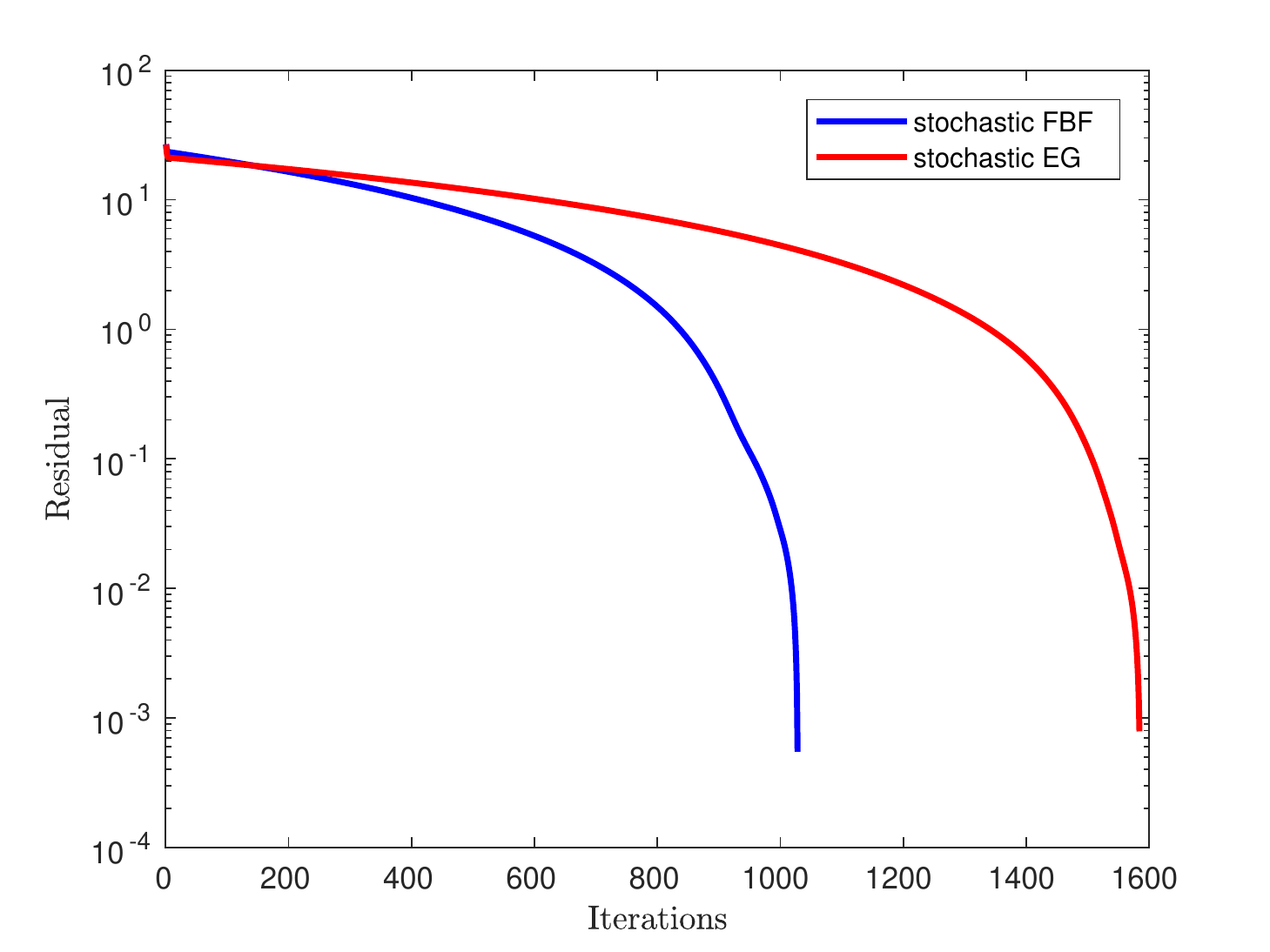}     		
  	\caption{Comparison between \ac{SFBF} and \ac{SEG} for solving the asymmetric game.
  		We represent the Residual vs. running time (left) and number of iterations (right) for one random example  $n_I=1000, n_{II}=2000$.}\label{Asymmetric_Fig1}
  \end{figure}

  \begin{figure}[ht!]
  	\centering
  	\includegraphics [width=0.75\textwidth]{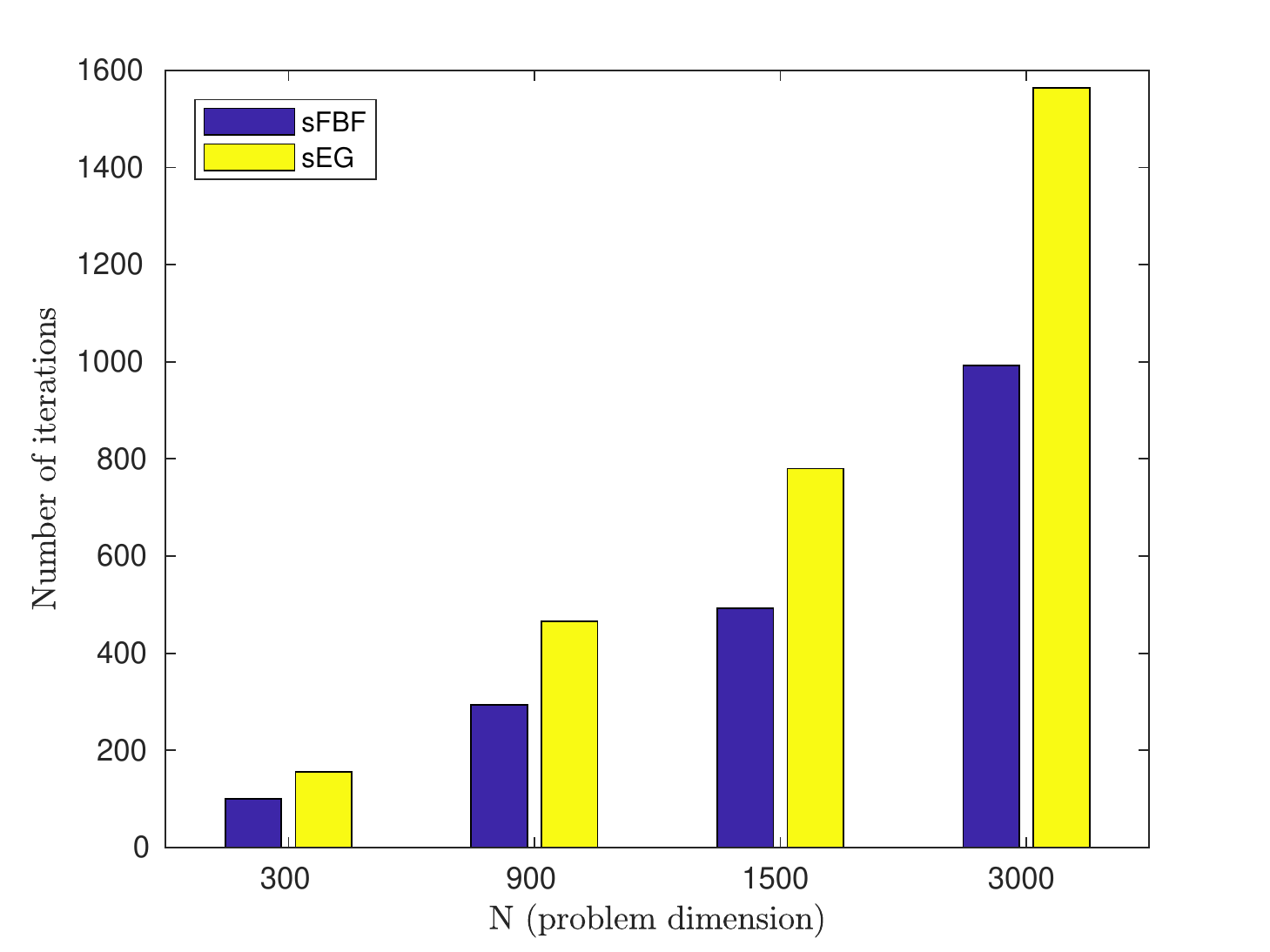}   		
  	\caption{Comparison on number of iterations between \ac{SFBF} and \ac{SEG} for solving the asymmetric game.
  		}\label{Asymmetric_Fig2}
  \end{figure} 
  
  \end{experiment}

\section{Conclusion}
\label{sec:conclusion}
In this paper we have developed a stochastic version of Tseng's forward-backward-forward algorithm for solving stochastic variational inequality problems over nonempty closed and convex sets. As in \cite{IusJofOliTho17}, the current analysis can be generalized to Cartesian $\VI$ problems, though have not done this explicitly. We show that the known theoretical convergence guarantees of  \ac{SEG} carry over to this setting, but our method consistently outperforms \ac{SEG} in terms of convergence rate and complexity. We therefore believe that \ac{SFBF} is a serious competitor to \ac{SEG} in typical primal-dual settings, where feasibility is a minor issue. Interesting directions for the future are to test the performance of the method in other instances where variance reduction is of importance, such as in composite optimization involving a large but finite sum of functions. Another possible extenstion would be to develop an infinite-dimensional Hilbert space version of the algorithm, and modify the basic \ac{SFBF} scheme to induce strong convergence of the iterates. We will investige these, and other issues, in the future.

\appendix
\section{Auxiliary Results}
\subsection{Proof of \cref{lem:Approx}}\label{app:A}
We start with a general result. Let $N\in\N$ and $\xi^{(1)},\ldots,\xi^{(N)}$ be an i.i.d sample from the measure $\measP$. Define the process $\left(M_{i}^{N}(x)\right)_{i=0}^{N}$ by $M_{0}(x)\eqdef 0$, and
for $1 \leq i \leq N$, by
\begin{equation}
M_{i}^{N}(x)\eqdef \frac{1}{N}\sum_{n=1}^{i}\left(F(x,\xi^{(n)})-T(x)\right)\qquad \forall x\in\mathbb{R}^{d}.
\end{equation}
Setting $\scrG_{i}\eqdef\sigma(\xi^{(1)},\ldots,\xi^{(i)}),1\leq i\leq N$, we see that the process $\{(M^{N}_{i}(x),\scrG_{i}),1\leq i\leq N\}$ is a martingale starting at zero.
\begin{lemma}\label{lem:M}
Let $p\geq 2$ be as specified in \cref{ass:variance}. For all $1\leq q \leq p,N\in\N$ and $x\in\mathbb{R}^{d}$, we have 
\begin{equation}
\Ex\left[\norm{M_{N}^{N}(x)}^{q}\right]^{\frac{1}{q}}\leq \frac{C_{q}}{\sqrt{N}}(\sigma(x^{\ast})+\sigma_{0}\norm{x-x^{\ast}}). 
\end{equation}
\end{lemma}
\begin{proof}
For $i\in\{1,2,\ldots,N\}$, the monotonicity of $L^{p}(\Pr)$ norms implies that 
\begin{align*}
\Ex\left[\norm{\Delta M_{i-1}^{N}(x)}^{q}\right]^{\frac{1}{q}}&=\frac{1}{N}\Ex\left[\norm{F(x,\xi^{(i)})-T(x)}^{q}\right]^{\frac{1}{q}}\\
&\leq \frac{1}{N}\Ex\left[\norm{F(x,\xi^{(i)})-T(x)}^{p}\right]^{\frac{1}{p}}\\
&\leq\frac{\sigma(x^{\ast})+\sigma_{0}\norm{x-x^{\ast}}}{N}.
\end{align*}
Using this, together with \cref{lem:BDG}, we get 
\begin{align*}
\Ex\left[\norm{M_{N}^{N}(x)}^{q}\right]^{1/q}&\leq C_{q}\sqrt{\sum_{k=1}^{N}\Ex\left(\left \| \frac{F(x,\xi^{(k)})-T(x)}{N} \right \|^q \right)^{2/q}}\\
&\leq C_{q} \sqrt{N^{-2}\sum_{k=1}^{N}\Ex\left(\norm{F(x,\xi^{(k)})-T(x)}^{q}\right)^{2/q}}\\
&\leq \frac{C_{q}(\sigma(x^{\ast})+\sigma_{0}\norm{x-x^{\ast}})}{\sqrt{N}}.
\end{align*}
\end{proof}

\begin{proof}[Proof of \cref{lem:Approx}]
Observe that $M^{m_{n+1}}_{m_{n+1}}(X_{n})=W_{n+1}$ and $M^{m_{n+1}}_{m_{n+1}}(Y_{n})=Z_{n+1}$. Hence, we immediately obtain from \cref{lem:M} that 
\begin{equation}
\Ex\left[\norm{W_{n+1}}^{p'}\vert\scrF_{n}\right]^{1/p'}\leq  \frac{C_{p'}(\sigma(x^{\ast})+\sigma_{0}\norm{X_{n}-x^{\ast}})}{\sqrt{m_{n+1}}}.
\end{equation}
To prove \eqref{eq:boundZ}, we notice that \cref{lem:M} implies that 
\begin{equation}\label{eq:Z1}
\Ex\left[\norm{Z_{n+1}}^{p'}\vert\hat{\scrF}_{n}\right]^{1/p'}\leq   \frac{C_{p'}(\sigma(x^{\ast})+\sigma_{0}\norm{Y_{n}-x^{\ast}})}{\sqrt{m_{n+1}}}. 
\end{equation}
The tower property of conditional expectations (recall that $\scrF_{n}\subseteq\hat{\scrF}_{n}$) gives 
\begin{align*}
\Ex\left[\norm{Z_{n+1}}^{p'}\vert\scrF_{n}\right]&=\Ex\left\{\Ex[\norm{Z_{n+1}}^{p'}\vert\hat{\scrF}_{n}]\vert\scrF_{n}\right\}\\
&\leq\left(\frac{C_{p'}}{\sqrt{m_{n+1}}}\right)^{p'} \Ex\left[\left(\sigma(x^{\ast})+\sigma_{0}\norm{Y_{n}-x^{\ast}}\right)^{p'}\vert\scrF_{n}\right]. 
\end{align*}
Finally, by the Minkowski inequality, we get  
\[
\Ex\left[\norm{Z_{n+1}}^{p'}\vert\scrF_{n}\right]^{1/p'}\leq \frac{C_{p'}}{\sqrt{m_{n+1}}} \left(\sigma(x^{\ast})+\sigma_{0}\Ex\left[\norm{Y_{n}-x^{\ast}}^{p'}\vert\scrF_{n}\right]^{1/p'}\right),
\]
and our proof is complete.
\end{proof}

\bibliographystyle{siam}
\bibliography{../bibtex/IEEEabrv,../bibtex/mybib,../bibtex/Bibliography}

\begin{thebibliography}{37}
\providecommand{\natexlab}[1]{#1}
\providecommand{\url}[1]{\texttt{#1}}
\expandafter\ifx\csname urlstyle\endcsname\relax
  \providecommand{\doi}[1]{doi: #1}\else
  \providecommand{\doi}{doi: \begingroup \urlstyle{rm}\Url}\fi

\bibitem[Atchad{\'e} et~al.(2017)Atchad{\'e}, Fort, and Moulines]{AtcForMou17}
Yves~F Atchad{\'e}, Gersende Fort, and Eric Moulines.
\newblock On perturbed proximal gradient algorithms.
\newblock \emph{J. Mach. Learn. Res}, 18\penalty0 (1):\penalty0 310--342, 2017.

\bibitem[Bottou et~al.(2018)Bottou, Curtis, and Nocedal]{BotCurNoc18}
L{\'e}on Bottou, Frank~E Curtis, and Jorge Nocedal.
\newblock Optimization methods for large-scale machine learning.
\newblock \emph{SIAM Review}, 60\penalty0 (2):\penalty0 223--311, 2018.

\bibitem[Boyd et~al.(2011)Boyd, Parikh, Chu, Peleato, and
  Eckstein]{BoyChuEckADMM11}
Stephen Boyd, Neal Parikh, Eric Chu, Borja Peleato, and Jonathan Eckstein.
\newblock Distributed optimization and statistical learning via the alternating
  direction method of multipliers.
\newblock \emph{Foundations and Trends{\textregistered} in Machine learning},
  3\penalty0 (1):\penalty0 1--122, 2011.

\bibitem[Boyd and Vandenberghe(2004)]{BV04}
Stephen~P. Boyd and Lieven Vandenberghe.
\newblock \emph{Convex Optimization}.
\newblock Cambridge University Press, Cambridge, UK, 2004.

\bibitem[Chen et~al.(2018)Chen, Chen, Ouyang, and Pasiliao]{CheCheOuPas18}
Chenxi Chen, Yunmei Chen, Yuyuan Ouyang, and Eduardo Pasiliao.
\newblock Stochastic accelerated alternating direction method of multipliers
  with importance sampling.
\newblock \emph{Journal of Optimization Theory and Applications}, 179\penalty0
  (2):\penalty0 676--695, 2018.
\newblock \doi{10.1007/s10957-018-1270-0}.
\newblock URL \url{https://doi.org/10.1007/s10957-018-1270-0}.

\bibitem[Combettes and Pesquet(2015)]{ComPes15}
P.~Combettes and J.~Pesquet.
\newblock Stochastic quasi-fej{\'e}r block-coordinate fixed point iterations
  with random sweeping.
\newblock \emph{SIAM Journal on Optimization}, 25\penalty0 (2):\penalty0
  1221--1248, 2018/09/20 2015.
\newblock \doi{10.1137/140971233}.
\newblock URL \url{https://doi.org/10.1137/140971233}.

\bibitem[Cui et~al.(2004)Cui, Goldsmith, and Bahai]{CGB04}
Shuguang Cui, Andrea~J. Goldsmith, and A.~Bahai.
\newblock Energy-efficiency of {MIMO} and cooperative {MIMO} techniques in
  sensor networks.
\newblock 22\penalty0 (6):\penalty0 1089--1098, August 2004.

\bibitem[Dang and Lan(2015)]{DanLan15}
Cong~D. Dang and Guanghui Lan.
\newblock On the convergence properties of non-euclidean extragradient methods
  for variational inequalities with generalized monotone operators.
\newblock \emph{Computational Optimization and Applications}, 60\penalty0
  (2):\penalty0 277--310, 2015.
\newblock \doi{10.1007/s10589-014-9673-9}.
\newblock URL \url{https://doi.org/10.1007/s10589-014-9673-9}.

\bibitem[Duflo(1996)]{Duf96}
Marie Duflo.
\newblock \emph{Algorithmes Stochastiques}.
\newblock Springer, New York, 1996.

\bibitem[Duvocelle et~al.(2018)Duvocelle, Mertikopoulos, Staudigl, and
  Vermeulen]{DuvMerStaVer18}
Benoit Duvocelle, Panayotis Mertikopoulos, Mathias Staudigl, and Dries
  Vermeulen.
\newblock Learning in time-varying games.
\newblock \emph{arXiv preprint arXiv:1809.03066}, 2018.

\bibitem[Facchinei and Pang(2003)]{FacPan03}
Francisco Facchinei and Jong-shi Pang.
\newblock \emph{Finite-Dimensional Variational Inequalities and Complementarity
  Problems - Volume I and Volume II}.
\newblock Springer Series in Operations Research, 2003.

\bibitem[Feng et~al.(2013)Feng, Jiang, Lim, {Cimini Jr.}, Feng, and
  Li]{FJLC+13}
Daquan Feng, Chenzi Jiang, Gubong Lim, Leonard~J. {Cimini Jr.}, Gang Feng, and
  Geoffrey~Ye Li.
\newblock A survey of energy-efficient wireless communications.
\newblock \emph{IEEE Communications Surveys \& Tutorials}, 15\penalty0
  (1):\penalty0 167--178, 2013.

\bibitem[Hata(1980)]{Hat80}
M.~Hata.
\newblock Empirical formula for propagation loss in land mobile radio services.
\newblock 29\penalty0 (3):\penalty0 317--325, August 1980.

\bibitem[Isheden et~al.(2012)Isheden, Chong, Jorswieck, and Fettweis]{ICJF12}
Christian Isheden, Zhijat Chong, Edward Jorswieck, and Gerhard Fettweis.
\newblock Framework for link-level energy efficiency optimization with informed
  transmitter.
\newblock 11\penalty0 (8):\penalty0 2946--2957, August 2012.

\bibitem[Iusem et~al.(2017)Iusem, Jofr{\'e}, Oliveira, and
  Thompson]{IusJofOliTho17}
AN~Iusem, Alejandro Jofr{\'e}, Roberto~I Oliveira, and Philip Thompson.
\newblock Extragradient method with variance reduction for stochastic
  variational inequalities.
\newblock \emph{SIAM Journal on Optimization}, 27\penalty0 (2):\penalty0
  686--724, 2017.

\bibitem[Jiang and Xu(2008)]{JiaXu09}
Houyuan Jiang and Huifu Xu.
\newblock Stochastic approximation approaches to the stochastic variational
  inequality problem.
\newblock \emph{IEEE Transactions on Automatic Control}, 53\penalty0
  (6):\penalty0 1462--1475, 2008.

\bibitem[Jofr{\'e} and Thompson(2018)]{JofTho18}
Alejandro Jofr{\'e} and Philip Thompson.
\newblock On variance reduction for stochastic smooth convex optimization with
  multiplicative noise.
\newblock \emph{Mathematical Programming}, 2018.
\newblock \doi{10.1007/s10107-018-1297-x}.
\newblock URL \url{https://doi.org/10.1007/s10107-018-1297-x}.

\bibitem[Juditsky et~al.(2011)Juditsky, Nemirovski, and Tauvel]{JNT11}
Anatoli Juditsky, Arkadi~Semen Nemirovski, and Claire Tauvel.
\newblock Solving variational inequalities with stochastic mirror-prox
  algorithm.
\newblock \emph{Stochastic Systems}, 1\penalty0 (1):\penalty0 17--58, 2011.

\bibitem[Kannan and Shanbhag(2012)]{KanShan12}
A.~Kannan and U.~Shanbhag.
\newblock Distributed computation of equilibria in monotone nash games via
  iterative regularization techniques.
\newblock \emph{SIAM Journal on Optimization}, 22\penalty0 (4):\penalty0
  1177--1205, 2017/12/28 2012.
\newblock \doi{10.1137/110825352}.
\newblock URL \url{https://doi.org/10.1137/110825352}.

\bibitem[King and Rockafellar(1993)]{KinRoc93}
Alan~J King and R~Tyrrell Rockafellar.
\newblock Asymptotic theory for solutions in statistical estimation and
  stochastic programming.
\newblock \emph{Mathematics of Operations Research}, 18\penalty0 (1):\penalty0
  148--162, 1993.

\bibitem[Korpelevich(1976)]{Kor76}
G.~M. Korpelevich.
\newblock The extragradient method for finding saddle points and other
  problems.
\newblock \emph{{\`E}konom. i Mat. Metody}, 12:\penalty0 747--756, 1976.

\bibitem[Kushner and Yin(1997)]{KusYin97}
H.~J. Kushner and G.~G. Yin.
\newblock \emph{Stochastic Approximation Algorithms and Applications}.
\newblock Springer, New York, 1997.

\bibitem[Mertikopoulos and Belmega(2016)]{MB16}
Panayotis Mertikopoulos and E.~Veronica Belmega.
\newblock Learning to be green: Robust energy efficiency maximization in
  dynamic {MIMO}-{OFDM} systems.
\newblock 34\penalty0 (4):\penalty0 743 -- 757, April 2016.

\bibitem[Mertikopoulos and Staudigl(2018)]{MerSta18b}
Panayotis Mertikopoulos and Mathias Staudigl.
\newblock Stochastic mirror descent dynamics and their convergence in monotone
  variational inequalities.
\newblock \emph{Journal of Optimization Theory and Applications}, 179\penalty0
  (3):\penalty0 838--867, December 2018.

\bibitem[Mertikopoulos and Zhou(2018)]{MerZho18}
Panayotis Mertikopoulos and Zhengyuan Zhou.
\newblock Learning in games with continuous action sets and unknown payoff
  functions.
\newblock \emph{Mathematical Programming}, 2018.
\newblock \doi{10.1007/s10107-018-1254-8}.
\newblock URL \url{https://doi.org/10.1007/s10107-018-1254-8}.

\bibitem[Polyak(1987)]{Pol87}
Boris~Teodorovich Polyak.
\newblock \emph{Introduction to Optimization}.
\newblock Optimization Software, New York, NY, USA, 1987.

\bibitem[Ravat and Shanbhag(2011)]{RavSha11}
U.~Ravat and U.~Shanbhag.
\newblock On the characterization of solution sets of smooth and nonsmooth
  convex stochastic nash games.
\newblock \emph{SIAM Journal on Optimization}, 21\penalty0 (3):\penalty0
  1168--1199, 2017/12/29 2011.
\newblock \doi{10.1137/100792644}.
\newblock URL \url{https://doi.org/10.1137/100792644}.

\bibitem[Rosasco et~al.(2016)Rosasco, Villa, and V{\~u}]{RosVilVu16}
Lorenzo Rosasco, Silvia Villa, and Bang~C{\^o}ng V{\~u}.
\newblock Stochastic forward--backward splitting for monotone inclusions.
\newblock \emph{Journal of Optimization Theory and Applications}, 169\penalty0
  (2):\penalty0 388--406, 2016.

\bibitem[Scutari et~al.(2010)Scutari, Palomar, Facchinei, and Pang]{Scu10}
Gesualdo Scutari, Daniel~P Palomar, Francisco Facchinei, and Jong-shi Pang.
\newblock Convex optimization, game theory, and variational inequality theory.
\newblock \emph{IEEE Signal Processing Magazine}, 27\penalty0 (3):\penalty0
  35--49 

\bibitem[Shapiro et~al.(2009)Shapiro, Dentcheva, and
  Ruszczy{\'n}ski]{DenRusSha09}
Alexander Shapiro, Darinka Dentcheva, and Andrzej 
  Ruszczy{\'n}ski.
\newblock \emph{Lectures on stochastic programming: modeling and theory}.
\newblock SIAM, 2009.

\bibitem[Shen and Yu(2018)]{SheWei2018a}
Kaiming Shen and Wei Yu.
\newblock Fractional programming for communication systems---part i: Power
  control and beamforming.
\newblock \emph{IEEE Transactions on Signal Processing}, 66\penalty0
  (10):\penalty0 2616--2630, 2018.

\bibitem[Solodov and Svaiter(1999)]{SolSva99}
M.~Solodov and B.~Svaiter.
\newblock A new projection method for variational inequality problems.
\newblock \emph{SIAM Journal on Control and Optimization}, 37\penalty0
  (3):\penalty0 765--776, 2019/01/30 1999.
\newblock \doi{10.1137/S0363012997317475}.
\newblock URL \url{https://doi.org/10.1137/S0363012997317475}.

\bibitem[Stroock(2011)]{Str11}
Daniel~W. Stroock.
\newblock \emph{Probability Theory: An Analytic View}.
\newblock Cambridge University Press, Cambridge, 2nd edition, 2011.
\newblock ISBN 0-521-43123-9.

\bibitem[Telatar(1999)]{Tel99}
I.~Emre Telatar.
\newblock Capacity of multi-antenna {Gaussian} channels.
\newblock \emph{European Transactions on Telecommunications and Related
  Technologies}, 10\penalty0 (6):\penalty0 585--596, 1999.

\bibitem[Tseng(2000)]{Tse00}
P.~Tseng.
\newblock A modified forward-backward splitting method for maximal monotone
  mappings.
\newblock \emph{SIAM Journal on Control and Optimization}, 38\penalty0
  (2):\penalty0 431--446, 2018/09/13 2000.
\newblock \doi{10.1137/S0363012998338806}.
\newblock URL \url{https://doi.org/10.1137/S0363012998338806}.

\bibitem[Von~Stengel(2002)]{VonStengel02}
Bernhard Von~Stengel.
\newblock Computing equilibria for two-person games.
\newblock \emph{Handbook of game theory with economic applications},
  3:\penalty0 1723--1759, 2002.

\bibitem[Yousefian et~al.(2017)Yousefian, Nedi{\'c}, and Shanbhag]{YouNedSha17}
Farzad Yousefian, Angelia Nedi{\'c}, and Uday~V. Shanbhag.
\newblock On smoothing, regularization, and averaging in stochastic
  approximation methods for stochastic variational inequality problems.
\newblock \emph{Mathematical Programming}, 165\penalty0 (1):\penalty0 391--431,
  2017.
\newblock \doi{10.1007/s10107-017-1175-y}.
\newblock URL \url{https://doi.org/10.1007/s10107-017-1175-y}.

\end{thebibliography}

\end{document}